\documentclass[reqno]{amsart}
\usepackage{amsmath,amssymb,stmaryrd}
\usepackage{amsrefs}
\usepackage{amsfonts}

\usepackage[colorlinks=true, pdfborder={0 0 0}]{hyperref}
\usepackage{cleveref}
\usepackage{upref}

\newtheorem{theorem}{Theorem}[section]
\newtheorem{defi}{Definition}[section]
\newtheorem{proposition}{Proposition}[section]

\newtheorem{lemma}{Lemma}[section]
\newtheorem{conj}{Conjecture}[section]

\def \rr {\mathbb{R}}
\def \rn {\mathbb{R}^n}
\def \nn {\mathbb{N}}
\def \rn {\mathbb{R}^n}

\def \eps {\epsilon}
\def \crit {2^\star}
\def \Pl {{\mathcal P}_\mu}
\def \Mmx {\Omega\setminus\{0\}}
\def \ul {u_\lambda}
\def \Ul {U_\lambda}
\def \vl {v_\lambda}
\def \xl {x_\lambda}
\def \nl {\nu_\lambda}

\def \yl {y_\lambda}
\def \ml {\mu_\lambda}

\def \l {\lambda}
\def \tpl {\tilde{\psi}_\mu}

\def \mU {{\mathcal U}}
\def \mV {{\mathcal V}}
\def \hundeux {H_{k,0}^2(\Omega)}

\def \bi {\begin{itemize}}
\def \ei {\end{itemize}}
\def \be {\begin{enumerate}}
\def \ee {\end{enumerate}}

\def \beq {\begin{equation}}
\def \eeq {\end{equation}}

\def \bt {\begin{theorem}}
\def \et {\end{theorem}}
\def \bp {\begin{proposition}}
\def \ep {\end{proposition}}
\def \bl {\begin{lemma}}
\def \el {\end{lemma}}
\def \bpr {\begin{proof}}
\def \epr {\end{proof}}
\def \beqn {\begin{eqnarray}}
\def \eeqn {\end{eqnarray}}

\def \dkdeux {D_k^2(\rn)}

\def \O {\Omega}
 \def \Obar {\overline{\Omega}}
 \def \Bbar {\overline{B}}
 
\title[Pucci-Serrin conjecture]{Critical dimensions for polyharmonic operators: The Pucci-Serrin conjecture for solutions of bounded energy}

\author{Fr\'ed\'eric Robert}
\address{Fr\'ed\'eric Robert, Institut \'Elie Cartan, Universit\'e de Lorraine, CNRS, IECL, F-54000 Nancy, France}
\email{frederic.robert@univ-lorraine.fr}

\date{Februray 21st, 2025}

\subjclass[2020]{Primary 35J35, Secondary 35J60, 35B44, 35J08}

 \addtocontents{toc}{\protect\setcounter{tocdepth}{1}}

\begin{document}
\begin{abstract} We prove a Pucci-Serrin conjecture on critical dimensions under a uniform bound on the energy. The method is based on the analysis of the Green's function of polyharmonic operators with "almost" Hardy potential.
\end{abstract}
\maketitle
\section{Introduction}
Let $B$ be the unit ball of $\rn$ and let $k\in\nn$ be such that $n>2k\geq 2$. Consider $\lambda\in\rr$ and $u\in C^{2k}(\overline{B})$ such that
\begin{equation}\label{eq:1}\left\{\begin{array}{cc}
\Delta^k u-\lambda u=|u|^{\crit-2}u&\hbox{ in }B\\
u=\partial_\nu u=...=\partial_{\nu}^{k-1}u=0&\hbox{ on }\partial B
\end{array}\right\}
\end{equation}
where $\crit:=\frac{2n}{n-2k}$.  A very interesting conjecture of Pucci and Serrin (\cite{pucci-serrin}, p58) is stated as follows:
\begin{conj} Let $B$ be the unit ball of $\rn$ and let $k\in\nn$ be such that $n>2k\geq 2$. Assume that 
$$2k<n<4k.$$
Then there exists $\lambda_0(n,k)>0$ such that for all $0<\lambda<\lambda_0(n,k)$, any radial solution to \eqref{eq:1} is identically null.
\end{conj}
Edmunds-Fortunato-Janelli \cite{EFJ} and Grunau \cite{grunau:1995} proved that there exists a positive radial solution to \eqref{eq:1} for all $\l\in (0,\l_1)$ when $n>4k$, where $\l_1>0$ is the first eigenvalue of $\Delta^k$ on $B$ with Dirichlet boundary condition. In particular, the expected range $(2k,4k)$ is optimal.  In this paper, we prove the following:
\begin{theorem}\label{th:1} Let $B$ be the unit ball of $\rn$ and let $k\in\nn$ be such that $n>2k\geq 2$. Assume that 
$$2k<n<4k.$$
Then, for any $M>0$, there exists $\lambda_0(n,k,M)>0$ such that for all $0<\lambda<\lambda_0(n,k,M)$, any radial solution to \eqref{eq:1} satisfying that $\Vert u\Vert_{\crit}\leq M$ is identically null.
\end{theorem}
Concerning terminology, Pucci-Serrin defined that a dimension $n>2k$ is {\it critical} if there exists $\lambda_0(n,k)>0$ such that any radial solution of \eqref{eq:1} is identically null when $0<\lambda<\lambda_0(n,k)$. Theorem \ref{th:1} proves the conjecture under any arbitrary fixed bound on the Lebesgue's norm.

\medskip\noindent Here is a brief history of the problem. The conjecture turns to be true in the following situations:
\begin{itemize}
\item $k=1$ (Brézis-Nirenberg \cite{brezis-nirenberg});
\item $k=2$ (Pucci-Serrin \cite{pucci-serrin});
\item $k\geq 2$ et  $n=2k+1$ (Pucci-Serrin \cite{pucci-serrin});
\item $k\geq 2$ et $2k<n<2k+6$ (Bernis-Grunau \cite{bernis-grunau} and Grunau \cite{grunau:II}).
\end{itemize}
In these situations, the proofs are based on Pohozaev-type identities for radial functions. The larger $k$ is, the trickier and longer the computations are and achieving $n<2k+6$ is a true "tour de force". Moreover, beside the computational difficulties, the methods in these papers do not seem enough to tackle the full conjecture (see Grunau \cite{grunau:II} for discussions on this issue).

\smallskip\noindent The case of positive functions is interesting in itself. Grunau \cite{grunau:positive} proved the validity of the conjecture when restricted to positive functions ({\it weakly critical dimensions}). In this situation, the key is to test a solution $u$ to \eqref{eq:1} against a carefully chosen positive polyharmonic function on $\Bbar$. The case of arbitrary sign-changing solutions involved in the original conjecture, the one we address here, is much more involved.

\smallskip\noindent As a final remark, we mention that Jannelli \cite{jannelli} has formalized the notion of {\it critical dimensions} in a more general setting by connecting it to the $L^2-$integrability of the Green's function.

\medskip\noindent In the present paper, we adopt a new approach that is based on the concentration analysis of families of solutions to \eqref{eq:1}: this permits to develop a method that is uniform and independent of the value of the power $k$. This approach is particularly relevant due to the critical exponent $\crit$ that may tolerate an unbounded family of solutions as $\l\to 0$: in this situation, this family should concentrate along explicit profiles referred to as bubbles. The general theory for second-order problems ($k=1$) has been performed in Druet-Hebey-Robert \cite{DHR} for positive solutions and was based on the comparison principle, see also Hebey \cite{hebey:zurich} for a modern point of view on such issues. We refer also to Druet-Laurain \cite{dl} regarding a method for positive solutions and to Premoselli \cite{premoselli} for a more recent and promising approach for sign-changing solutions. We also refer to Carletti \cite{carletti} for a beautiful asymptotic analysis when $k>1$.   

\smallskip\noindent Due to the sign-change and to the lack of comparison principle when $k\geq 2$, we develop tools based on Green's representation formula for a linear equations. More precisely, we rewrite \eqref{eq:1} as $Pu=0\, +\,\{\hbox{bdy conditions}\}$ where $P=\Delta^k-\l-|u|^{\crit-2}$ and we express $u$ in terms of the Green's function of $P$. The core and the bulk of our analysis is to get a sharp pointwise control of this Green's function, which is the object of Theorem \ref{APP:GREEN:th:Green:pointwise:BIS}. This control is based on the regularity Lemma \ref{APP:GREEN:lem:main} for solutions to linear equations with "almost" Hardy-type potential.

\smallskip\noindent The hypothesis on radial symmetry is essential. In addition to prescribing concentration at the center of the ball, radiality forces solutions of $\Delta^k u=|u|^{\crit-2}u$ on $\rn$ to have a fixed sign, see Theorem \ref{APP:RAD:th:1} below. This does not happen in the non-radial case as shown for instance by Molica Bisci and Pucci \cite{MBP}.

\smallskip\noindent Most of the analysis is valid for any elliptic operator like $\Delta^k+...$: the restriction $n<4k$ and the specificity of $\Delta^k-\lambda$ are used only for the final argument involving the Pohozaev-Pucci-Serrin identity. We will make an intensive use of the elliptic regularity of the reference Agmon-Douglis-Nirenberg \cite{ADN}. For the convenience of the reader, the last section \ref{APP:GREEN:sec:regul:adn} is a collection of results contained in \cite{ADN}.

\smallskip\noindent {\it Notations:} $C(a,b,...)$ will denote any constant depending only on $a,b,...$. The same notation might refer different constants from line to line, and even in the same line.

\smallskip\noindent{\it Acknowledgement:} The author thanks Emmanuel Hebey for remarks and comments on this work.

\section{Preliminary analysis}
We prove Theorem \ref{th:1} by contradiction. We fix $M>0$. If Theorem \ref{th:1} is not true, then there exists a sequence $(\lambda_i)_{i\in\nn}\in\rr_{>0}$ and $(u_i)_{i\in\nn}\in C^{2k}(\Bbar)$ radially symmetrical such that

\begin{equation}\label{eq:2}\left\{\begin{array}{cc}
\Delta^k u_i-\lambda_i u=|u_i|^{\crit-2}u_i&\hbox{ in }B\\
u_i=\partial_\nu u_i=...=\partial_{\nu}^{k-1}u_i=0&\hbox{ on }\partial B\\
u_i\not\equiv 0&\\
\Vert u_i\Vert_{\crit}\leq M&\\
\lim_{i\to\infty}\lambda_i=0&
\end{array}\right\}
\end{equation}
In order to simplify the exposition, we assume that there exists $\lambda_0>0$ such that for all $0<\lambda<\lambda_0$, there exists $\ul\in C^{2k}(\Bbar)$ radially symmetrical such that
\begin{equation}\label{eq:3}\left\{\begin{array}{cc}
\Delta^k \ul-\lambda \ul=|\ul|^{\crit-2}\ul&\hbox{ in }B\\
\ul=\partial_\nu \ul=...=\partial_{\nu}^{k-1}\ul=0&\hbox{ on }\partial B\\
\ul\not\equiv 0&\\
\Vert \ul\Vert_{\crit}\leq M&
\end{array}\right\}
\end{equation}
We are performing an analysis of $\ul$ as $\lambda\to 0$. All the results and statements will be up to the extraction of subfamilies, although we will always refer to $\ul$. A preliminary  remark is that $\ul\in C^{2k+1,\theta}(B)$, $0<\theta<1$, due to elliptic regularity.

\subsection{Sobolev spaces and inequalities}
For any $\Omega\subset \rn$ a smooth domain, $p\geq 1$ and $l\in\nn$, we define $H_l^p(\Omega)$ (resp. $H_{l,0}^p(\Omega)$) as the completion of  $\{u\in C^\infty(\Omega)\hbox{ s.t. }\Vert u\Vert_{H_l^p}<\infty\}$ (resp. $C^\infty_c(\Omega)$) for the norm $u\mapsto \Vert u\Vert_{H_l^p}:=\sum_{i\leq l}\Vert \nabla^i u\Vert_p$. Given a finite set $S\subset\Omega$, we define $L_{loc}^p(\Omega\setminus S)=\{u:\Omega\to \rr\hbox{ s.t. }\eta u\in L^p(\Omega)\hbox{ for all }\eta\in C^\infty_c(\rn\setminus S)\}$,  $H_{l,loc}^p(\Omega\setminus S)=\{u:\Omega\to \rr\hbox{ s.t. }\eta u\in H_{l}^p(\Omega)\hbox{ for all }\eta\in C^\infty_c(\rn\setminus S)\}$ and $H_{l,0,loc}^p(\Omega\setminus S)=\{u:\Omega\to \rr/\,\hbox{ s.t. }\eta u\in H_{l,0}^p(\Omega)\hbox{ for all }\eta\in C^\infty_c(\rn\setminus S)\}$. This notation is a bit abusive since $\Omega\setminus S$ is open, but there will be no ambiguity in this paper. In the specific case $p=2$ and $\Omega$ is bounded, note that on $\hundeux$,  $\Vert \cdot\Vert_{H^2_k}$ is equivalent to  $u\mapsto \Vert \Delta^{k/2}u\Vert_2$. Here and in the sequel, $\Delta^{\frac{i}{2}}=\nabla\Delta^{\frac{i-1}{2}}$ when $i$ is odd. Note that for $u\in C^{2k}(\Obar)$ and $\Omega$ a smooth bounded domain of $\rn$ or $\Omega$ is a half-space, then $\{u\in H_{k,0}^2(\Omega)\}\Leftrightarrow \{u=\partial_\nu u=...=\partial_{\nu}^{k-1}u=0\hbox{ on }\partial \O\}$.

\medskip\noindent We let $\dkdeux$ be the completion of $C^\infty_c(\rn)$ for the norm $u\mapsto \Vert \Delta^{k/2}u\Vert_2$. It follows from Sobolev's theorem that there exists $K(n,k)>0$ such that
\begin{equation}\label{sobo:ineq:rn}
\left(\int_{\rn}|u|^{\crit}\, dx\right)^\frac{2}{\crit}\leq K(n,k)\int_{\rn}(\Delta^{\frac{k}{2}}u)^2\, dx\hbox{ for all }u\in\dkdeux.
\end{equation}
As one checks, this inequality is a also valid for all $u\in H_{k,0}^2(\O)$.

\begin{lemma} Let $(\ul)_{\lambda>0}\in C^{2k}(\Bbar)$ be a family radially symmetrical solution to \eqref{eq:3}. Then $\lim_{\lambda\to 0}\Vert\ul\Vert_\infty=+\infty$.
\end{lemma}
\begin{proof} We argue by contradiction. If the conclusion does not hold, then there exists $C>0$ such that $\Vert\ul\Vert_\infty\leq C$ for all $\lambda>0$. It follows from elliptic theory (Theorems \ref{APP:GREEN:th:2:again} and \ref{APP:GREEN:th:2:holder}) that $\Vert\ul\Vert_{C^{2k,1/2}}\leq C$ for all $\lambda>0$. It then follows from Ascoli's theorem that there exists $u_0\in C^{2k}(\Bbar)$ such that $\lim_{\lambda\to 0}\ul=u_0$ in $C^{2k}(\Bbar)$. Passing to the limit in \eqref{eq:3} yields
\begin{equation}\label{eq:57}\left\{\begin{array}{cc}
\Delta^k u_0=|u_0|^{\crit-2}u_0&\hbox{ in }B\\
u_0=\partial_\nu u_0=...=\partial_{\nu}^{k-1}u_0=0&\hbox{ on }\partial B
\end{array}\right\}
\end{equation}
It then follows from Lazzo-Schmidt (point (a) of Corollary 3.10 of \cite{ls}) that $u_0\equiv 0$.

\smallskip\noindent Multiplying \eqref{eq:3} by $\ul$, integrating by parts and using H\"older's inequality yield
\begin{eqnarray*}
\int_B(\Delta^{k/2}\ul)^2\, dx=\int_B\ul\Delta^k\ul\, dx=\lambda\int_B\ul^2\, dx+\int_B|\ul|^{\crit}\, dx\leq C\lambda\Vert \ul\Vert_{\crit}^2+\Vert\ul\Vert_{\crit}^{\crit}.  
\end{eqnarray*}
With the Sobolev inequality \eqref{sobo:ineq:rn} and using that $\ul\not\equiv 0$ and $\ul\in H_{k,0}^2(B)$, we get that $K(n,k)^{-1}\leq C\lambda+\Vert\ul\Vert_{\crit}^{\crit-2}$. Passing to the limit $\lambda\to 0$ and using that $u_0\equiv 0$, we get a contradiction. This proves the Lemma.\end{proof}

\smallskip\noindent Note that as a consequence of the preceding argument, $(\ul)_\l$ is bounded in $H_{k,0}^2(B)$, that is there exists $C(M)>0$ such that $\Vert \ul\Vert_{H_k^2}\leq C(M)$ for all $\l>0$.

\begin{lemma}\label{lem:mass:0} Let $(\yl)_\lambda\in B$ and $(r_\l)_{\l>0}\in\rr_{>0}$ be such that $\lim_{\l\to 0}r_\l^{-1}|\yl|=+\infty$. Then
$$\lim_{\l\to 0}\int_{B_{r_\l}(\yl)\cap B}|\ul|^{\crit}\, dx=0.$$
\end{lemma}
\begin{proof} Let us fix $N\in\nn$. There exists a group of isometries of $\rn$, say $G$, such that $\sharp G \geq N$ and there exists $\eps_N>0$ such that $d(\sigma(e_1),\tau(e_1))\geq \eps_N$ for all $\sigma,\tau\in G$, $\sigma\neq \tau$. Here, $e_1$ is the first vector of the canonical basis of $\rn$. Therefore, as one checks, $B_{r_\l}(\sigma(\yl))\cap B_{r_\l}(\tau(\yl))=\emptyset$ for all $\sigma,\tau\in G$, $\sigma\neq \tau$ and $\l>0$ is small enough. With  the invariance of $\ul$ under the action of the group $G$, we get that
\begin{eqnarray*}
M^{\crit}&\geq &\int_B |\ul|^{\crit}\, dx\geq \int_{\bigcup_{\sigma\in G}B_{r_\l}(\sigma(\yl))\cap B}|\ul|^{\crit}\, dx\\
& \geq &\sum_{\sigma\in G}\int_{B_{r_\l}(\sigma(\yl))\cap B}|\ul|^{\crit}\, dx=\sharp G\int_{B_{r_\l}(\yl)\cap B}|\ul|^{\crit}\, dx 
\end{eqnarray*}
and therefore 
$$\int_{B_{r_\l}(\yl)\cap B}|\ul|^{\crit}\, dx\leq \frac{M^{\crit}}{N}\hbox{ as }\lambda\to 0.$$
Since this is valid for all $N$, the conclusion follows.\end{proof}

\begin{lemma}\label{lem:weak:est} Let $(\ul)_{\lambda>0}\in C^{2k}(\Bbar)$ be a family radially symmetrical solution to \eqref{eq:3}. Then there exists $C>0$ such that $|x|^{\frac{n-2k}{2}}|\ul(x)|\leq C$ for all $x\in B$ and $\l\to 0$.
\end{lemma}
\begin{proof} We prove the lemma by contradiction. We set $w_\l(x):=|x|^{\frac{n-2k}{2}}|\ul(x)|$ for all $x\in B$ and $\l>0$. Let us assume that
$$w_\l(\yl):=\sup_{x\in B}w_\l(x)\to +\infty\hbox{ as }\l\to 0.$$
We define $r_\l:=|\ul(\yl)|^{-\frac{2}{n-2k}}$. We have that
\begin{equation}\label{eq:yl:rl}
\frac{|\yl|}{r_\l}=w_\l(\yl)^{\frac{2}{n-2k}}\to \infty\hbox{ and }r_\l\to 0\hbox{ as }\l\to 0.
\end{equation}
\smallskip\noindent{\it Case 1:} assume that 
\begin{equation}\label{lim:bndy:1}
\lim_{\l\to 0}\frac{d(\yl,\partial B)}{r_\l}=+\infty.
\end{equation}
We define $$\vl(x):=r_\l^{\frac{n-2k}{2}}\ul(\yl+r_\l x)\hbox{ for }x\in \frac{B-\yl}{r_\l}.$$
A change of variable in \eqref{eq:3} yields
\beq\label{eq:vl}
\Delta^k\vl-\l r_\l^{2k}\vl=|\vl|^{\crit-2}\vl\hbox{ in }\frac{B-\yl}{r_\l}.
\eeq
It follows from the definition of $\yl$ that
$$|\yl+r_\l x|^{\frac{n-2k}{2}}|\ul(\yl+r_\l x)|\leq |\yl|^{\frac{n-2k}{2}}|\ul(\yl)|\hbox{ for }x\in \frac{B-\yl}{r_\l},$$
and then
$$\left|\frac{\yl}{|\yl|}+\frac{r_\l}{|\yl|} x\right|^{\frac{n-2k}{2}}|\vl(x)|\leq 1\hbox{ for }x\in \frac{B-\yl}{r_\l}.$$
We fix $R>0$. It  follows from \eqref{lim:bndy:1} and the above inequality that there exists $\lambda_R>0$ such that
\begin{equation*}
B_R(0)\subset \frac{B-\yl}{r_\l}\hbox{ and }|\vl(x)|\leq 2\hbox{ for all }x\in B_R(0)\hbox{ and }0<\l<\l_R.
\end{equation*}
With \eqref{eq:vl}, it then follows from elliptic theory (Theorems \ref{APP:GREEN:th:2:again} and \ref{APP:GREEN:th:2:holder}) and Ascoli's theorem that there exists $v\in C^{2k}(\rn)$ such that $\lim_{\l\to 0}\vl=v$ in $C^{2k}_{loc}(\rn)$. Given $R>0$, with a change of variable, we get that
$$\int_{B_R(0)}|\vl|^{\crit}\, dx=\int_{B_{R r_\l}(\yl)}|\ul|^{\crit}\, dx.$$
It follows from Lemma \ref{lem:mass:0} and \eqref{eq:yl:rl} that passing to the limit yields $\int_{B_R(0)}|v|^{\crit}\, dx=0$ for all $R>0$, so that $v\equiv 0$ since it is continuous. However, since $|\vl(0)|=1$, we get that $|v(0)|=1$, which contradicts $v\equiv 0$. This ends Case 1.

\medskip\noindent {\it Case 2:}
$$\lim_{\l\to 0}\frac{d(\yl,\partial B)}{r_\l}=\rho\in [0,+\infty).$$
Up to a rotation, we then get that
$$\lim_{\l\to 0}\frac{B-\yl}{r_\l}=(-\infty, \rho)\times \rr^{n-1}.$$
The proof is then similar to Case 1 by working on this half-space. We leave the details to the reader. This yields also to a contradiction.

\smallskip\noindent In both cases, we have gotten a contradiction, which proves the Lemma.\end{proof}

\begin{lemma}\label{lem:cv:out:0} Let $(\ul)_{\lambda>0}\in C^{2k}(\Bbar)$ be a family radially symmetrical solution to \eqref{eq:3}. Then  $\lim_{\l\to 0}\ul=0$ in $C^{2k}_{loc}(\Bbar\setminus\{0\})$.
\end{lemma}
\begin{proof} It follows from Lemma \ref{lem:weak:est} that for all $\delta>0$, there exists $C(\delta)>0$ such that $|\ul(x)|\leq C(\delta)$ for all $\l>0$ and $x\in B\setminus B_\delta(0)$. It follows from elliptic theory (Theorems \ref{APP:GREEN:th:2:again} and \ref{APP:GREEN:th:2:holder}) and Ascoli's theorem that there exists $u_0\in C^{2k}(\Bbar\setminus\{0\})$ such that $\lim_{\l\to 0}\ul=u_0$ in $C^{2k}_{loc}(\Bbar\setminus\{0\})$. Since $\Vert \ul\Vert_{H_k^2}\leq C(M)$ for all $\l>0$, we also get that $u_0\in H_{k,0}^2(B)$ and $\ul\rightharpoonup u_0$ weakly in $H_{k,0}^2(B)$. Passing to the limit $\l\to 0$ in \eqref{eq:3}, we get that $u_0$ is a weak solution to \eqref{eq:57}. Regularity theory (see Van der Vorst \cite{vdv} and Theorems \ref{APP:GREEN:th:2:again} and \ref{APP:GREEN:th:2:holder}) yields $u_0\in C^{2k}(\Bbar)$ is a strong solution to \eqref{eq:57}, and then $u_0\equiv 0$ by \cite{ls} since it is radial. This proves the Lemma.\end{proof}

\smallskip\noindent We will make use of the following classification:
\bt[Swanson \cite{swanson}] \label{APP:RAD:th:1} Let $k,n\in\mathbb{N}$ be such $2\leq 2k<n$. Let $u\in \dkdeux$ be a distributional solution to $\Delta^k u=|u|^{\crit-2}u$ in $\rn$. Assume that $u$ is radially symmetric. Then there exists $\mu>0$ and $\eps\in\{-1,0,+1\}$ such that
$$u(x)=\eps \left(\frac{\mu}{\mu^2+a_{n,k}|x|^2}\right)^{\frac{n-2k}{2}},\hbox{ where }a_{n,k}:=\left(\Pi_{j=-k}^{k-1}(n+2j)\right)^{-\frac{1}{k}}.$$ 
\et
\begin{proof} Although Swanson's Theorem 4 in \cite{swanson} is only stated for positive  functions, the proof is working for any functions. More precisely, if $u(0)\neq 0$, we follow exactly Swanson's proof. If $u(0)=0$, the arguments of Swanson (Lemma 7) yield $u\equiv 0$.\end{proof}

\begin{lemma}\label{lem:add:bump} Let $(\yl)_\l\in B$ be such that $\lim_{\l\to 0}|\yl|^{\frac{n-2k}{2}}|\ul(\yl)|=c\in (0,+\infty)$. Then there exists $(r_\l)_\l\in (0,+\infty)$ such $\lim_{\l\to 0}r_\l=0$, $\lim_{\l\to 0}r_\l^{-1}|\yl|=c'\in (0,+\infty)$ and
\beq
\lim_{\l\to 0}r_\l^{\frac{n-2k}{2}}\ul(r_\l \cdot)=\eps U\hbox{ in }C^{2k}_{loc}(\rn\setminus\{0\}),
\eeq
for some $\eps\in \{-1,+1\}$ where 
\beq\label{def:U}
U(x)=\left(\frac{1}{1+a_{n,k}|x|^2}\right)^{\frac{n-2k}{2}}\hbox{ for all }x\in\rn.
\eeq
\end{lemma}
\begin{proof}
It follows from Lemma \ref{lem:cv:out:0} that $\yl\to 0$ as $\l\to 0$. We set $s_\l:=|\yl|$ and we define $W_\l(x):=s_\l^{\frac{n-2k}{2}}\ul(s_\l x)$ for $x\in B_{1/s_\l}(0)$ and $\l>0$. Lemma \ref{lem:weak:est} yields
\beq\label{bnd:vl}
|W_\l(x)|\leq C|x|^{-\frac{n-2k}{2}}\hbox{ for all }x\in B_{1/s_\l}(0)\hbox{ and }\l>0.
\eeq
A change of variable in \eqref{eq:3} yields
\beq\label{eq:vl:2}
\Delta^k W_\l-\l s_\l^{2k}W_\l=|W_\l|^{\crit-2}W_\l\hbox{ in }B_{1/s_\l}(0).
\eeq
Due to elliptic theory (Theorems \ref{APP:GREEN:th:2:again} and \ref{APP:GREEN:th:2:holder}) and Ascoli's theorem, \eqref{bnd:vl} and \eqref{eq:vl:2} yield the existence of $W\in C^{2k}(\rn\setminus\{0\})$ such that
$$\lim_{\l\to 0}W_\l=W\hbox{ in }C^{2k}_{loc}(\rn\setminus\{0\}).$$
Since $W_\l\left(\frac{\yl}{|\yl|}\right)=|\yl|^{\frac{n-2k}{2}}\ul(\yl)$, passing to the limit $\l\to 0$ yields $|W(Y_0)|=c>0$ where $Y_0:=\lim_{\l\to 0}\frac{\yl}{|\yl|}$. Therefore $W\not\equiv 0$.

\smallskip\noindent   We   prove that $W\in \dkdeux$. Let us fix $l\in \{0,...,k\}$. It follows from Sobolev's embedding that there exists $C(l,k,n)>0$ such that
\beq\label{ineq:sobo:l}
\left(\int_{B}|\nabla^l\varphi|^{\crit(l)}\, dx\right)^{\frac{2}{\crit(l)}}\leq C(l,k,n)\int_{B}(\Delta^{k/2}\varphi)^2\, dx
\eeq
for all $\varphi\in H_{k,0}^2(B)$, where $\crit(l):=\frac{2n}{n-2(k-l)}$. Given $R>0$, with a change of variable, we get\begin{eqnarray*}
\left(\int_{B_R(0)\setminus B_{R^{-1}}(0)}|\nabla^lW_\l|^{\crit(l)}\, dx\right)^{\frac{2}{\crit(l)}}&=&\left(\int_{B_{Rr_\l}(0)\setminus B_{R^{-1}r_\l}(0)}|\nabla^l\ul|^{\crit(l)}\, dx\right)^{\frac{2}{\crit(l)}}\\
&\leq& C(l,k,n)\int_{B}(\Delta^{k/2}\ul)^2\, dx\leq C
\end{eqnarray*}
since $\Vert\ul\Vert_{H_k^2}$ is uniformly bounded. Letting $\l\to 0$ and $R\to +\infty$ yields $|\nabla^l W|\in L^{\crit(l)}(\rn)$. We now let $\eta\in C^\infty_c(\rn)$ be such that $\eta(x)=1$ for $x\in B_1(0)$ and $\eta(x)=0$ for $x\in \rn\setminus B_2(0)$. For $R>0$, we define $W_R(x):=\left(1-\eta(Rx)\right)\eta(\frac{x}{R})W(x)$ for all $x\in\rn$. Since $|\nabla^lW|\in L^{\crit(l)}(\rn)$ for all $l\in\{0,...,k\}$, one gets that $(W_R)_R$ is a Cauchy family in $\dkdeux$ as $R\to +\infty$, so it has a limit in $\dkdeux$ as $R\to +\infty$, and then $W\in \dkdeux$. So Theorem \ref{APP:RAD:th:1} yields the existence of  $t>0$ and $\eps\in \{-1,+1\}$ such that
$$W(x)=\eps \left(\frac{t}{t^2+a_{n,k}|x|^2}\right)^{\frac{n-2k}{2}}\hbox{ for all }x\in\rn.$$
Therefore, setting $r_\l:=t s_\l$, we get the conclusion of the Lemma.\end{proof}

\section{Sharp analysis at the furthest scale}
\begin{proposition}\label{lem:nl} Let $(\ul)_\l\in C^{2k}(\Bbar)$ be a family of solutions to \eqref{eq:3}. Then there exists $(\nl)_\l\in (0,+\infty)$ and $\eps_0\in \{-1,+1\}$ such that 
\begin{equation*}
\lim_{\l\to 0}\nl=0;
\end{equation*}
\beq\label{eq:22}
\lim_{\l\to 0}\nl^{\frac{n-2k}{2}}\ul(\nl \cdot)=\eps_0 U\hbox{ in }C^{2k}_{loc}(\rn\setminus\{0\});
\eeq
\begin{equation*}
\lim_{R\to +\infty}\lim_{\l\to 0}\sup_{x\in B\setminus B_{R\nl}(0)}|x|^{\frac{n-2k}{2}}|\ul(x)|=0.
\end{equation*}
\end{proposition}
\begin{proof} Given $N\geq 1$, we say that $({\mathcal H}_N)$ holds if there exists $(\mu_{\l,1})_\l,...,(\mu_{\l,N})_\l\in (0,+\infty)$ such that
$$\lim_{\l\to 0}\frac{\mu_{\l,i}}{\mu_{\l,i+1}}=0\hbox{ for all }i=1,...,N-1\hbox{ and }\lim_{\l\to 0}\mu_{\l,N}=0,$$
and that for all $i\in \{1,...,N\}$, there exists $\eps_i\in \{-1,+1\}$ such that
$$\lim_{\l\to 0}v_{\l,i}=\eps_i U\hbox{ in }C^{2k}_{loc}(\rn\setminus\{0\})\hbox{ where }v_{\l,i}(x):=\mu_{\l,i}^{\frac{n-2k}{2}}\ul(\mu_{\l,i}x)\hbox{ for all }x\in B_{1/\mu_{\l,i}}(0),$$
while for $i=1$, this convergence holds in $C^{2k}_{loc}(\rn)$. 

\medskip\noindent {\bf Step 1:} We claim that $({\mathcal H}_1)$ holds. 

\smallskip\noindent We prove the claim. We define $\xl\in B$ and $\mu_{\l,1}:=\ml:=|\ul(\xl)|^{-\frac{2}{n-2k}}$ where $|\ul(\xl)|=\sup_B|\ul|$. We define
\beq\label{def:Ul}
\Ul(x):=\ml^{\frac{n-2k}{2}}\ul(\ml x)\hbox{ for all }x\in B_{1/\ml}(0).
\eeq
It then follows from elliptic theory (Theorems \ref{APP:GREEN:th:2:again} and \ref{APP:GREEN:th:2:holder}) that there exists $\tilde{U}\in C^{2k}(\rn)$ such that $\lim_{\l\to 0}\Ul=\tilde{U}$ in $C^{2k}_{loc}(\rn)$ and $\Delta^k \tilde{U}=|\tilde{U}|^{\crit-2}\tilde{U}$. The definiton of $\ml$ and Lemma \ref{lem:weak:est} yield $|\xl|\le C\ml$, so there exists $X_0\in\rn$ such that $\lim_{\l\to 0}\frac{\xl}{\ml}=X_0$. We have that $|\Ul(\frac{\xl}{\ml})|=1$, so that, letting $\l\to 0$ yields $|\tilde{U}(X_0)|=1$. Therefore $\tilde{U}\not\equiv 0$ and $|\tilde{U}|\leq |\tilde{U}(x_0)|=1$. As in Lemma \ref{lem:add:bump}, we get that $\tilde{U}\in\dkdeux$ and Theorem  \ref{APP:RAD:th:1} yields the conclusion.

\medskip\noindent {\bf Step 2:} Assume that $({\mathcal H}_N)$ holds for some $N\geq 1$ and that
\beq\label{hyp:contrad}
\lim_{R\to +\infty}\lim_{\l\to 0}\sup_{x\in B\setminus B_{R\mu_{\l,N}}(0)}|x|^{\frac{n-2k}{2}}|\ul(x)|>0.
\eeq
Then $({\mathcal H}_{N+1})$ holds. 

\smallskip\noindent We prove the claim. It follows from \eqref{hyp:contrad} that there exists $(\yl)_\l\in B$ such that $\lim_{\l\to 0}\frac{|\yl|}{\mu_{\l,N}}=+\infty$ and $\lim_{\l\to 0}|\yl|^{\frac{n-2k}{2}}|\ul(\yl)|=c>0$. We define $\mu_{\l, N+1}:=r_\l$, where $r_\l>0$ is given by Lemma \ref{lem:add:bump}. As one checks, we get that $({\mathcal H}_{N+1})$ holds. The claim is proved.

\medskip\noindent {\bf Step 3:} We claim that there exists $C(M,n,k)>0$ such that if $({\mathcal H}_N)$ holds, then $N\leq C(M,n,k)$.

\smallskip\noindent We prove the claim. For any $i\in \{1,...,N\}$, we get that
\begin{eqnarray*}
\lim_{R\to +\infty}\lim_{\l\to 0}\int_{B_{R\mu_{\l,i}}(0)\setminus B_{R^{-1}\mu_{\l,i}}(0)}|\ul|^{\crit}\, dx&=&\lim_{R\to +\infty}\lim_{\l\to 0}\int_{B_{R}(0)\setminus B_{R^{-1}}(0)}|v_{\l,i}|^{\crit}\, dx\\
&=&\int_{\rn}U^{\crit}\, dx
\end{eqnarray*}
Since the $N$ domains $B_{R\mu_{\l,i}}(0)\setminus B_{R^{-1}\mu_{\l,i}}(0)$ are distinct for $\l\to 0$, we get that
$$\sum_{i=1}^N\int_{B_{R\mu_{\l,i}}(0)\setminus B_{R^{-1}\mu_{\l,i}}(0)}|\ul|^{\crit}\, dx=\int_{\bigcup_i B_{R\mu_{\l,i}}(0)\setminus B_{R^{-1}\mu_{\l,i}}(0)}|\ul|^{\crit}\, dx\leq \int_{B}|\ul|^{\crit}\, dx\leq M^{\crit}.$$
And then $N\leq C(M,n,k)$ with $C(M,n,k):=\frac{M^{\crit}}{\int_{\rn}U^{\crit}\, dx}$. This proves the claim.

\medskip\noindent{\bf Step 4:} We conclude the proof of the Proposition. We let $N\geq 1$ be maximal such that $({\mathcal H}_N)$ holds: the existence follows from Step 2. It  follows from Step 1 that
$$\lim_{R\to +\infty}\lim_{\l\to 0}\sup_{x\in B\setminus B_{R\mu_{\l,N}}(0)}|x|^{\frac{n-2k}{2}}|\ul(x)|=0.$$
Therefore Proposition \ref{lem:nl} follows by taking $\nl:=\mu_{\l, N}$. \end{proof}

\begin{proposition}\label{prop:gamma} Let $(\ul)_\l\in C^{2k}(\Bbar)$ be a family of solutions to \eqref{eq:3}, and let $(\nl)_\l$ be as in Proposition \ref{lem:nl}. Then for any $\gamma\in (0,n-2k)$, there exists $C>0$ such that
\beq\label{ineq:gamma}
|\ul(x)|\leq C\frac{\nl^{\frac{n-2k}{2}-\gamma}}{|x|^{n-2k-\gamma}}\hbox{ for all }x\in B\setminus B_{\nl}(0)\hbox{ and }\l\to 0.
\eeq

\end{proposition}

\begin{proof} We fix $R>0$ and we define $V_\l:={\bf 1}_{B\setminus B_{R\nl}(0)}|\ul|^{\crit-2}$ so that 
$$(\Delta^k-\l-V_\l)\ul=0\hbox{ in }B\setminus B_{2R\nl}(0).$$
Let $\mu_\gamma>0$  be as in the statement of Theorem \ref{APP:GREEN:th:Green:pointwise:BIS}. It follows from Proposition \ref{lem:nl} that there exists $R=R_\gamma>0$ such that
\begin{equation*}
|X|^{2k}|\l+V_\l(x)|\leq \mu_\gamma\hbox{ for all }x\in B\hbox{ and }\l>0\hbox{ small enough}.
\end{equation*}
We let $G_\l$ be the Green's function for $\Delta^k-\l-V_\l$ on $B$ with Dirichlet boundary condition given by Theorem \ref{APP:GREEN:th:Green:main} with the pointwise controls of Theorem \ref{APP:GREEN:th:Green:pointwise:BIS}. We choose $x\in B$ such that $|x|>3R\nl$. Since $(\Delta^k-\l-V_\l)\ul=0$, we get that

\begin{eqnarray*}
\ul(x)&=&\int_{B\setminus B_{2R\nl}(0)} G_\l(x,\cdot)(\Delta^k-\l-V_\l)\ul\, dy\\
&&+\int_{\partial(B\setminus B_{2R\nl}(0))}\sum_{i=0}^{k-1}\left(\partial_\nu\Delta^{i}\ul\Delta^{k-1-i}G_\l(x,\cdot)- \Delta^{i}\ul\partial_\nu\Delta^{k-1-i}G_\l(x,\cdot)\right)\, d\sigma\\
&=&\sum_{i=0}^{k-1}\int_{\partial B_{2R\nl}(0)}\nabla^{1+2i}\ul\star \nabla^{2(k-1-i)}_yG_\l(x,\cdot)+\nabla^{2i}\ul\star \nabla^{1+2(k-1-i)}_yG_\l(x,\cdot)
\end{eqnarray*}
where $T\star S$ denotes any linear combination of contractions of the tensors $T$ and $S$. For all $j=0,...,2k-1$, it follows from the convergence  \eqref{eq:22} that 
$$|\nabla^{j}\ul(y)|\leq C \nl^{-\frac{n-2k}{2}-j}\hbox{ for }y\in \partial B_{2R\nl}(0).$$
The pointwise controls of Theorem \ref{APP:GREEN:th:Green:pointwise:BIS} and \eqref{APP:GREEN:ineq:36:bis} yield
$$|\nabla_y^jG(x,y)|\leq C|y|^{-\gamma-j} |x|^{2k-n+\gamma}\hbox{ for all }x\in B\setminus B_{3R\nl}(0)\hbox{ and }y\in \partial B_{2R\nl}(0).$$
Therefore, we get that
\begin{eqnarray*}
\ul(x)&\leq &C \nl^{\frac{n-2k}{2}-\gamma} |x|^{2k-n+\gamma}\hbox{ for all }x\in B\setminus B_{3R\nl}(0).
\end{eqnarray*}
The validity of this inequality on $B_{3R\nl}(0)\setminus B_{\nl}(0)$ is a consequence of \eqref{eq:22}. This proves Proposition \ref{prop:gamma}.\end{proof}

\begin{proposition} Let $(\ul)_\l\in C^{2k}(\Bbar)$ be a family of solutions to \eqref{eq:3}, and let $(\nl)_\l$ be as in Proposition \ref{lem:nl}. Then for any $\omega\subset\subset B$, there exists $C>0$ such that
\beq\label{upp:ul}
|\ul(x)|\leq C\frac{\nl^{\frac{n-2k}{2}}}{|x|^{n-2k}}\hbox{ for all }x\in \omega\setminus B_{\nl}(0)\hbox{ and }\l\to 0,
\eeq
and
\beq\label{lim:ul:green}
\lim_{\l\to 0}\frac{\ul}{\nl^{\frac{n-2k}{2}}}=H:=\eps_0\left(\int_{\rn}U^{\crit-1}\, dx\right) G_0(0,\cdot)\hbox{ in }C^{2k}_{loc}(B\setminus\{0\})
\eeq
where $G_0(0,\cdot)$ is the Green's function for $\Delta^k$ on $B$ with Dirichlet boundary condition. In particular $\Delta^k H=0$ in $B\setminus\{0\}$.
\end{proposition}
\begin{proof} Let us fix $x\in \omega$ such that $|x|>4\nl$. Let $G_\l$ be the Green's function of $\Delta^k-\l$ in $B$ with Dirichlet boundary condition. The existence follows from Theorem \ref{APP:GREEN:th:Green:main}. Green's representation formula yields
\begin{eqnarray}
\ul(x)&=& \int_B G_\l(x,y)|\ul(y)|^{\crit-2}\ul(y)\, dy= \int_{|x-y|>|x|/2}+\int_{|x-y|<|x|/2}\label{ineq:39}
\end{eqnarray}
We estimate these terms separately. Regarding the second term of \eqref{ineq:39}, for $y\in B$ such that $|x-y|<|x|/2$, we have that $|y|>|x|/2>2\nl$, and we apply \eqref{ineq:gamma} and we use \eqref{APP:GREEN:ctrl:G} to get
\begin{eqnarray}
&&\left|\int_{|x-y|<|x|/2}G_\l(x,y)|\ul(y)|^{\crit-2}\ul(y)\, dy\right|\nonumber\\
&&\leq  C\int_{|x-y|<|x|/2}|x-y|^{2k-n}\frac{\nl^{(\frac{n-2k}{2}-\gamma)(\crit-1)}}{ |x|^{(n-2k-\gamma)(\crit-1)}}\, dy\leq C \frac{\nl^{\frac{n-2k}{2}}}{|x|^{n-2k}}\cdot \left(\frac{\nl}{ |x|}\right)^{2k-\gamma(\crit-1)}\label{ineq:38}
\end{eqnarray}
We split the first term of \eqref{ineq:39} in three parts. First, using the pointwise control \eqref{APP:GREEN:ctrl:G}, H\"older's inequality and $\Vert \ul\Vert_{\crit}\leq M$, we get
\begin{eqnarray}
&&\left|\int_{\{|x-y|<|x|/2\}\cap \{|y|<R^{-1}\nl\}}G_\l(x,y)|\ul(y)|^{\crit-2}\ul(y)\, dy\right|\nonumber\\
&&\leq \int_{\{|x-y|<|x|/2\}\cap \{|y|<R^{-1}\nl\}}|x-y|^{2k-n}|\ul(y)|^{\crit-1}\, dy\nonumber\\
&&\leq  C |x|^{2k-n}\int_{B_{R^{-1}\nl}(0)}|\ul(y)|^{\crit-1}\, dy\nonumber\\
&&\leq C|x|^{2k-n}\left(\int_{B_{R^{-1}\nl}(0)}\, dy\right)^{\frac{1}{\crit}}\left(\int_{B_{R^{-1}\nl}(0)}|\ul(y)|^{\crit}\, dy\right)^{\frac{\crit-1}{\crit}}\nonumber\\
&&\leq    C|x|^{2k-n}\left( R^{-1}\nl \right)^{\frac{n-2k}{2}}= CR^{-\frac{n-2k}{2}}\frac{\nl^{\frac{n-2k}{2}}}{|x|^{n-2k}}\label{ineq:37}
\end{eqnarray}
Now, using \eqref{APP:GREEN:ctrl:G} similarly and the pointwise control \eqref{ineq:gamma}, we get
 \begin{eqnarray}
&&\left|\int_{\{|x-y|<|x|/2\}\cap \{|y|\geq R\nl\}}G_\l(x,y)|\ul(y)|^{\crit-2}\ul(y)\, dy\right|\nonumber\\
&&\leq \int_{\{|x-y|<|x|/2\}\cap \{|y|\geq R\nl\}}|x-y|^{2k-n}|\ul(y)|^{\crit-1}\, dy\nonumber\\
&&\leq  C |x|^{2k-n}\int_{B\setminus B_{R\nl}(0)}\frac{ \nl^{(\frac{n-2k}{2}-\gamma)(\crit-1)}}{ |y|^{(n-2k-\gamma)(\crit-1)}}\, dy\leq C \frac{|x|^{2k-n} \nl^{\frac{n-2k}{2}}}{R^{2k-\gamma(\crit-1)}}\label{ineq:36}
\end{eqnarray}
for $\gamma<\frac{2k}{\crit-1}$. Taking $R=1$ and plugging \eqref{ineq:38}, \eqref{ineq:37}, \eqref{ineq:36}  in \eqref{ineq:39}, we get \eqref{upp:ul}.

\smallskip\noindent We fix $x\in B$ such that $x\neq 0$, so that all the preceding estimates hold. For any $R>0$, with a change of variable, we have that

\begin{eqnarray*}
&&\int_{\{|x-y|<|x|/2\}\cap B_{R\nl}\setminus B_{R^{-1}\nl}(0)}G_\l(x,y)|\ul(y)|^{\crit-2}\ul(y)\, dy\\
&&=\int_{B_{R\nl}\setminus B_{R^{-1}\nl}(0)}G_\l(x,y)|\ul(y)|^{\crit-2}\ul(y)\, dy\\
&&=\nl^{\frac{n-2k}{2}}\int_{B_{R}\setminus B_{R^{-1}}(0)}G_\l(x,\nl z)|\Ul(z)|^{\crit-2}\Ul(z)\, dz.
\end{eqnarray*}
Independently, given $x\in B\setminus\{0\}$, the definition and uniqueness of Green's functions of Theorem \ref{APP:GREEN:th:green:nonsing} combined with the integral bound \eqref{APP:GREEN:eq:16} yields the convergence of $(G_\lambda(x,\cdot))_\lambda$ to $G_0(x,\cdot)$ uniformly in $C^{0}_{loc}(B\setminus\{x\})$ as $\l\to 0$. Therefore,  \eqref{eq:22} yields
\begin{eqnarray*}
&&\lim_{R\to +\infty}\lim_{\l\to 0}\nl^{-\frac{n-2k}{2}}\int_{\{|x-y|<|x|/2\}\cap B_{R\nl}\setminus B_{R^{-1}\nl}(0)}G_\l(x,y)|\ul(y)|^{\crit-2}\ul(y)\, dy\\
&&=\eps_0\left(\int_{\rn}U^{\crit-1}\, dx\right)G_0(0,x).
\end{eqnarray*}
Combining this latest limit with \eqref{ineq:38}, \eqref{ineq:37}, \eqref{ineq:36} and \eqref{ineq:39}, we get the pointwise limit in \eqref{lim:ul:green}. The convergence in $C^{2k}$ is consequence of elliptic theory (Theorems \ref{APP:GREEN:th:2:again} and \ref{APP:GREEN:th:2:holder}). This ends the proof of the Proposition.\end{proof}

\section{Conclusion via the Pohozaev-Pucci-Serrin identity}\label{sec:poho}
The following identities are essentially in Pucci-Serrin \cite{pucci-serrin-var-id} and are generalizations of the historical Pohozaev identity \cite{pohozaev}. We recall  them for the sake of completeness. The first lemma is a straightforward iteration:
\begin{lemma}\label{lem:deltap} For any $v\in C^\infty(\Omega)$, where $\Omega$ is a domain of $\rn$, we have that
\begin{equation*}
\left\{\begin{array}{cc}
\Delta ^p(x^i\partial_i v)=2p\Delta ^p v+x^i\partial_i\Delta ^p v&\hbox{ for all }p\in\nn\hbox{ and }\\
\partial_j\Delta ^p(x^i\partial_i v)=(2p+1)\partial_j\Delta ^p v+x^i\partial_i(\partial_j\Delta ^p v)&\hbox{ for all }p\in\nn\hbox{ and }j=1,...,n.
\end{array}\right\}.
\end{equation*}
These identities rewrite $\Delta ^\frac{l}{2}(x^i\partial_i v)=l\Delta ^\frac{l}{2} v+x^i\partial_i(\Delta ^\frac{l}{2} v)$ for all $l\in\nn$.
\end{lemma}
\begin{proposition}\label{prop:poho} Let $\Omega\subset\rn$ be a smooth bounded domain with $2\leq 2k<n$. Then for all $u\in C^{2k+1}(\rn)$ and $c\in\rr$, we have that
\begin{eqnarray*}
&&\int_\Omega\left(\Delta ^k u-c|u|^{\crit-2}u\right) T(u)\, dx=\int_{\partial\Omega}\left((x,\nu)\left(\frac{|\Delta^\frac{k}{2}u|^2}{2}-\frac{c|u|^{\crit}}{\crit}\right)+S(u)\right)\, d\sigma
\end{eqnarray*}
where $T(u):=\frac{n-2k}{2}u+x^i\partial_i u$ and
\begin{eqnarray}
S(u)&:=&\sum_{i=0}^{E(k/2)-1}\left(-\partial_\nu\Delta ^{k-i-1}u\Delta ^iT(u)+\Delta ^{k-i-1}u\partial_\nu\Delta ^iT(u)\right)\label{def:S}\\
&&-{\bf 1}_{\{k\hbox{ odd}\}}\partial_\nu(\Delta ^\frac{k-1}{2} u)\Delta ^\frac{k-1}{2}T(u)\nonumber
\end{eqnarray}
\end{proposition}
\begin{proof} Integrating by parts, for any $l\in\nn$, $l\geq 1$, $U,V\in C^{2l}(\rn)$, we have that
\begin{equation}\label{formula:ipp}
\int_{\Omega }(\Delta^l  U)V \, dX=\int_{\Omega }U(\Delta^l  V)\, dX+\int_{\partial\Omega}\mathcal{B}^{(l)}(U,V)\, d\sigma
\end{equation}
where
\begin{equation}\label{def:B}
\mathcal{B}^{(l)}(U,V):=\sum_{i=0}^{l-1}\left(-\partial_\nu\Delta ^{l-i-1}U\Delta ^i V+\Delta ^{l-i-1}U\partial_\nu\Delta ^iV\right)
\end{equation}
We first assume that $k=2p$ is even, with $p\in\nn$. Using Lemma \ref{lem:deltap}, we get 
\begin{eqnarray*}
&&\int_\Omega\left(\Delta ^k u-c|u|^{\crit-2}u\right) T(u)\, dx=\int_\Omega\Delta ^pu\Delta^p T(u)\, dx+\int_{\partial\Omega}\mathcal{B}^{(p)}(\Delta ^p u,T(u))\, d\sigma\\
&&-\left(\frac{n-2k}{2}\int_\Omega c |u|^{\crit}\, dx+\int_\Omega cx^i\frac{\partial_i |u|^{\crit}}{\crit}\, dx\right)\\
&&=\int_\Omega\Delta^p\left(\frac{n}{2}\Delta^p u+x^i\partial_i(\Delta ^p u)\right)\, dx+\int_{\partial\Omega}\mathcal{B}^{(p)}(\Delta ^p u,T(u))\, d\sigma\\
&&-\left(\frac{n-2k}{2}\int_\Omega c |u|^{\crit}\, dx-\frac{n}{\crit}\int_\Omega  c|u|^{\crit}\, dx+\int_{\partial\Omega}(x,\nu)\frac{c|u|^{\crit}}{\crit}\, dx\right)\\
&&=\int_\Omega\partial_i\left(\frac{x^i(\Delta ^p u)^2}{2}\right)\, dx+\int_{\partial\Omega}\mathcal{B}^{(p)}(\Delta ^p u,T(u))\, d\sigma-  \int_{\partial\Omega}(x,\nu)\frac{c|u|^{\crit}}{\crit}\, dx\\
&&=\int_{\partial\Omega}\left((x,\nu)\left(\frac{(\Delta ^p u)^2}{2}-\frac{c|u|^{\crit}}{\crit}\right)+\mathcal{B}^{(p)}(\Delta ^p u,T(u))\right)\, d\sigma
\end{eqnarray*}
which proves Proposition \ref{prop:poho} when $k$ is even. When $k=2q+1$ is odd, we get that
\begin{eqnarray*}
\int_\Omega \Delta ^k u T(u)\, dx&=& \int_\Omega \Delta ^q(\Delta ^{q+1}u)T(u)\, dx\\
&=& \int_\Omega\Delta ^{q+1}u\Delta ^q T(u)\, dx+\int_{\partial\Omega}\mathcal{B}^{(q)}(\Delta ^{q+1}u,T(u))\, d\sigma\\
&=& \int_\Omega \sum_j\partial_j(\Delta ^q u)\partial_j(\Delta ^q T(u))\, dx\\
&&+\int_{\partial\Omega}\left(\mathcal{B}^{(q)}(\Delta ^{q+1}u,T(u))-\partial_\nu(\Delta ^q u)\Delta ^qT(u)\right)\, d\sigma
\end{eqnarray*}
Using Lemma \ref{lem:deltap}, we get that
\begin{eqnarray*}
&&\int_\Omega \Delta ^k u T(u)\, dx= \int_\Omega \sum_j\partial_j(\Delta ^q u)\left(\frac{n}{2}\partial_j \Delta ^q u+x^i\partial_i\partial_j \Delta ^q u\right)\\
&&+\int_{\partial\Omega}\left(\mathcal{B}^{(q)}(\Delta ^{q+1}u,T(u))-\partial_\nu(\Delta ^q u)\Delta ^qT(u)\right)\, d\sigma\\
&&= \int_\Omega \partial_i \left(x^i\frac{(\partial_j\Delta ^q u)^2}{2}\right)\, dx +\int_{\partial\Omega}\left(\mathcal{B}^{(q)}(\Delta ^{q+1}u,T(u))-\partial_\nu(\Delta ^q u)\Delta ^qT(u)\right)\, d\sigma\\
&&= \int_{\partial\Omega}(x,\nu)\frac{|\nabla\Delta ^q u|^2}{2} \, d\sigma +\int_{\partial\Omega}\left(\mathcal{B}^{(q)}(\Delta ^{q+1}u,T(u))-\partial_\nu(\Delta^q u)\Delta ^qT(u)\right)\, d\sigma\end{eqnarray*}
Using the same computations as in the case when $k$ is even, we get the conclusion of Proposition \ref{prop:poho}.\end{proof}

We fix $\delta\in (0,1)$.  Since $\ul$ solves \eqref{eq:3}, Proposition \ref{prop:poho} yields
\beq
\l \int_{B_\delta(0)}\ul T(\ul)\, dx=\int_{\partial B_\delta(0)}\left((x,\nu)\left(\frac{|\Delta^\frac{k}{2}\ul|^2}{2}-\frac{|\ul|^{\crit}}{\crit}\right)+S(\ul)\right)\, d\sigma\label{eq:34}
\eeq
where $T(\ul)$ and $S(\ul)$ are as in \eqref{def:S}.
\begin{proposition} Let $(\ul)_\l\in C^{2k}(\Bbar)$ be a family of solutions to \eqref{eq:3}, and let $(\nl)_\l$ as in Proposition \ref{lem:nl}. Fix $0<\delta<1$. Then there exists $A\in\rr$ such that
\beq\label{eq:37}
\lim_{\l\to 0}\nl^{2k-n}\int_{\partial B_\delta(0)}\left((x,\nu)\left(\frac{|\Delta^\frac{k}{2}\ul|^2}{2}-\frac{|\ul|^{\crit}}{\crit}\right)+S(\ul)\right)\, d\sigma=A< 0.
\eeq
\end{proposition}
\begin{proof} Setting $\bar{u}_\l:=\nl^{-\frac{n-2k}{2}}\ul$ and using \eqref{lim:ul:green}, we get that
\begin{eqnarray*}
&&\int_{\partial B_\delta(0)}\left((x,\nu)\left(\frac{|\Delta^\frac{k}{2}\ul|^2}{2}-\frac{|\ul|^{\crit}}{\crit}\right)+S(\ul)\right)\, d\sigma\\
&&=\nl^{n-2k}\int_{\partial B_\delta(0)}\left((x,\nu)\left(\frac{|\Delta^\frac{k}{2}\bar{u}_\l|^2}{2}-\nl^{2k}\frac{|\bar{u}_\l|^{\crit}}{\crit}\right)+S(\bar{u}_\l)\right)\, d\sigma\\
&&=\nl^{n-2k}\left(\int_{\partial B_\delta(0)}\left((x,\nu)\left(\frac{|\Delta^\frac{k}{2}H|^2}{2}\right)+S(H)\right)\, d\sigma+o(1)\right)
\end{eqnarray*}
It follows from Boggio's formula \cite{boggio} (see also Lemma 2.27 in \cite{ggs}) that there exists $A_{k,n}>0$ such that
$$G_0(x,0)=A_{k,n}|x|^{2k-n}\int_1^{1/|x|}(v^2-1)^{k-1}v^{1-n}\, dv$$
for all $x\in B\setminus\{0\}$. Therefore, since $\eps_0\int_{\rn}U^{\crit-1}\, dx\neq 0$, there exists $\beta\in C^{2k}(\Bbar)$ and $\alpha\neq 0$ such that
\begin{equation*}
H(x)=\alpha\left(\Gamma(x)+\beta(x)\right)\hbox{ for all }x\in B\setminus \{0\}\hbox{, }\Gamma(x):=|x|^{2k-n}\hbox{ and }\beta(0)<0.
\end{equation*}
We then get that
\beq\label{eq:36}
\lim_{\l\to 0}\nl^{2k-n}\int_{\partial B_\delta(0)}\left((x,\nu)\left(\frac{|\Delta^\frac{k}{2}\ul|^2}{2}-\frac{|\ul|^{\crit}}{\crit}\right)+S(\ul)\right)\, d\sigma=C_\delta
\eeq
$$\hbox{where }C_\delta:=\alpha^2\int_{\partial B_\delta(0)}\left((x,\nu)\left(\frac{|\Delta^\frac{k}{2}(\Gamma+\beta)|^2}{2}\right)+S(\Gamma+\beta)\right)\, d\sigma.$$
Applying Proposition \ref{prop:poho} to $\Gamma+\beta$ on $B_\delta(0)\setminus B_r(0)$ for $0<r<\delta$ with $\eps=0$ and $f\equiv 0$, we get that $C_\delta$ is independent of the choice of $0<\delta<1$. We compute the different terms of $C_\delta$ separately. Using that $\beta$ and all its derivatives are bounded in $\Bbar$, we get that
\begin{eqnarray*}
\int_{\partial B_\delta(0)}(x,\nu)\frac{|\Delta^\frac{k}{2}(\Gamma+\beta)|^2}{2}\, d\sigma
&=& \int_{\partial B_\delta(0)}(x,\nu)\frac{|\Delta^\frac{k}{2}\Gamma|^2}{2}\, d\sigma+O\left( \delta^{k}\right)+O(\delta^{n})
\end{eqnarray*}
We let $S_P$ be the natural bilinear form such that $S(u)=S_P(u,u)$ for all $u$. With the expression \eqref{def:S} of $S$, we get that
\begin{eqnarray*}
S(\Gamma+\beta)&=&S_P(\Gamma,\Gamma)+S_P(\Gamma,\beta)+S_P(\beta,\Gamma)+S_P(\beta,\beta)\\
&=& S(\Gamma)-(\partial_\nu\Delta ^{k-1}\Gamma) T(\beta)+O(|x|^{2-n})
\end{eqnarray*}
With \eqref{APP:GREEN:calc:delta:gamma}, we get that
\begin{equation*}
 S(\Gamma+\beta)=S(\Gamma)+\frac{n-2k}{2\omega_{n-1}}\beta(0)|x|^{1-n}+O(|x|^{2-n})
\end{equation*}
These identities yield
\begin{equation*}
C_\delta=\alpha^2D_\delta +\frac{(n-2k)\alpha^2}{2\omega_{n-1}}\beta(0)+O(\delta)
\end{equation*}
where
$$D_r:=\int_{\partial B_{r}(0)}\left((x,\nu) \frac{|\Delta^\frac{k}{2} \Gamma|^2}{2}  +S( \Gamma)\right)\, d\sigma$$
for all $r>0$. Taking the identity of Proposition \ref{prop:poho} for $c=0$, $u\equiv\Gamma$ so that $\Delta^k u=0$ and $\Omega=B_1(0)-B_r(0)$ for $0<r<1$, we get that $D_r=D_1$ for all $0<r<1$. A quick computation yields  the existence of $D_{k,n}\in\rr$ such that
$$\left(|x| \frac{|\Delta^\frac{k}{2} \Gamma|^2}{2}  +S( \Gamma)\right)=D_{k,n}|x|^{2k+1-2n}\hbox{ for all }x\in \rn\setminus\{0\},$$
so that $D_r=D_{k,n}\omega_{n-1}r^{2k-n}$ for all $0<r<1$. Since this quantity is independent of $r$, we get that $D_{k,n}=0$, so that $D_\delta=0$ for all $\delta>0$ and then
$$C_\delta=\frac{(n-2k)\alpha^2}{2\omega_{n-1}}\beta(0)+O(\delta).$$
Since $C_\delta$ is independent of $\delta$, using \eqref{eq:36}, we get \eqref{eq:37} with $A=\frac{(n-2k)\alpha^2}{2\omega_{n-1}}\beta(0)<0$.\end{proof}

\begin{proposition} Let $(\ul)_\l\in C^{2k}(\Bbar)$ be a family of solutions to \eqref{eq:3}, and let $(\nl)_\l$ as in Proposition \ref{lem:nl}. Fix $0<\delta<1$. Then
\beq\label{eq:40}
\int_{B_\delta(0)}\ul T(\ul)\, dx=O(\nl^{n-2k})\hbox{ if }2k<n<4k.
\eeq
\end{proposition}
\begin{proof} Integrating by parts, we get that
\begin{eqnarray*}
\int_{B_\delta(0)}\ul T(\ul)\, dx&=&\int_{B_\delta(0)}\ul \left(\frac{n-2k}{2}\ul+x^i\partial_i\ul\right)\, dx\\
&=&-k\int_{B_\delta(0)}\ul^2\, dx+\int_{\partial B_\delta(0)}(x,\nu)\ul^2\, d\sigma
\end{eqnarray*}
With H\"older's inequality and the pointwise control \eqref{upp:ul}, using that $n<4k$, we get
\begin{eqnarray*}
\int_{B_\delta(0)}\ul^2\, dx&\leq & \int_{B_{\nl}(0)}\ul^2\, dx +\int_{B_\delta(0)\setminus B_{\nl}(0)}\nl^{n-2k}|x|^{2(2k-n)}\, dx\\
&\leq & \left(\int_{B_{\nl}(0)}\, dx\right)^{\frac{\crit-2}{\crit}} \left(\int_{B_{\nl}(0)}\ul^{\crit}\, dx\right)^{\frac{2}{\crit}}+\int_{B_\delta(0)\setminus B_{\nl}(0)}\nl^{n-2k}|x|^{2(2k-n)}\, dx\\
&\leq & C\nl^{2k}+\int_{B_\delta(0)\setminus B_{\nl}(0)}\nl^{n-2k}|x|^{2(2k-n)}\, dx\leq C\nl^{n-2k}
\end{eqnarray*}
since $n<4k$. The result then follows from these estimates and \eqref{lim:ul:green}.\end{proof}

\medskip\noindent{\bf Conclusion of the argument and proof of Theorem \ref{th:1}.} Putting \eqref{eq:37} and \eqref{eq:40} into the identity \eqref{eq:34}, we get that $o(\nl^{n-2k})=(A+o(1))\nl^{n-2k}$ as $\l\to 0$, which contradicts $A\neq 0$.

\section{Green's function for an "almost" Hardy operator}
Let $\Omega$ be a smooth domain of $\rn$ and let $k\in\nn$ be such that $2\leq 2k<n$. Given $h\in L^\infty(\O)$, we consider operators like $P=\Delta^k+h$. Integrating by parts yields $\int_\Omega u P u\, dx=\int_\Omega \left((\Delta^{\frac{k}{2}}u)^2+h u^2\right)\, dx$ for all $u\in C^\infty_c(\O)$, so that this expression  makes sense for $u\in \hundeux$. We say that $P$ is coercive if there exists $c>0$ such that $\int_\O uPu\, dx\geq c\Vert u\Vert_{H_k^2}^2$ for all $u\in\hundeux$. We prove the following theorems:
\begin{theorem}\label{APP:GREEN:th:Green:main} Let $\Omega$ be a smooth domain of $\rn$ such that $0\in\Omega$ is an interior point. Fix $k\in\nn$ such that $2\leq 2k<n$. We consider an operator $P=\Delta^k+h$, where $h\in L^\infty(\O)$ and $P$ is coercive. We let $V\in L^1(\Omega)$ such that for some $\mu>0$, $|x|^{2k}|V(x)|\leq \mu\hbox{ for all }x\in \Omega$. 

\smallskip\noindent Then there is $\mu_0(P,h)>0$ such that for $0<\mu<\mu_0(P,h)$, there exists $G:(\Mmx)\times (\Mmx)\setminus\{(z,z)/z\in \Mmx\}\to \rr$ such that:
\begin{itemize}
\item For all $x\in \Mmx$, $G(x,\cdot)\in L^q(\Omega)$ for all $1\leq q<\frac{n}{n-2k}$
\item For all $x\in \Mmx$, $G(x,\cdot)\in L^{\frac{2n}{n-2k}}_{loc}(\O\setminus\{x\})$
\item For all $f\in L^{\frac{2n}{n+2k}}(\Omega)\cap L^p_{loc}(\O\setminus\{0\})$, $p>\frac{n}{2k}$, we let $\varphi\in \hundeux$ such that $P\varphi=f$ in the weak sense. Then $\varphi\in C^0(\Obar\setminus\{0\})$ and
$$\varphi(x)=\int_\Omega  G(x,\cdot) f\, dy\hbox{ for all }x\in \Mmx.$$
\end{itemize}
Moreover, such a function $G$ is unique. It is the Green's function for $P-V$. In addition, $G$ is symmetric and for all $x\in \Mmx$, $$G(x,\cdot)\in H_{2k,loc}^p(\Omega\setminus\{0,x\})\cap H_{k,0,loc}^2(\O\setminus\{x\})\cap C^{2k-1}(\Obar\setminus\{0,x\})$$
for all $1<p<\infty$ and
\begin{equation*}
\left\{\begin{array}{cc}
(P-V)G(x,\cdot)=0&\hbox{ in }\O\setminus\{0,x\}\, \\
{\partial_{\nu}^iG(x,\cdot)}_{|\partial\O}=0&\hbox{ for }i=0,...,k-1.
\end{array}\right\}
\end{equation*}
\end{theorem}
In addition, we get the following pointwise control:
\begin{theorem}\label{APP:GREEN:th:Green:pointwise:BIS} Let $\Omega$ be a smooth domain of $\rn$ such that $0\in\Omega$ is an interior point. Fix $k\in\nn$ such that $2\leq 2k<n$, $L>0$ and $\mu>0$. We consider an operator $P=\Delta^k+h$, where $h\in L^\infty(\O)$, $\Vert h\Vert_{L^\infty}\leq L$ and $\int_\Omega uPu\,dx\geq L^{-1}\Vert u\Vert_{H_k^2}^2$ for all $u\in\hundeux$. We let $V\in L^1(\Omega)$ such that $P-V$ is coercive and
$$|x|^{2k}|V(x)|\leq \mu\hbox{ for all }x\in \Omega.$$
We let $G$ be the Green's function of $P-V$ as in Theorem \ref{APP:GREEN:th:Green:main}. 

\smallskip\noindent Then for any $\gamma\in (0,n-2k)$, there exists $\mu_\gamma>0$ such that for $\mu<\mu_\gamma$, for any $\omega\subset\subset\O$, for any $x\in\omega\setminus\{0\}$, $y\in \Mmx$ such that $x\neq y$, we have that 
\begin{itemize}
\item $$|G(x,y)|\leq C(\O,\gamma,L,\mu, k,\omega)\left(\frac{\max\{|x|,|y|\}}{\min\{|x|,|y|\}}\right)^\gamma |x-y|^{2k-n}$$
\item If $|x|<|y|$ and $l\leq 2k-1$, we have that
$$|\nabla_y^lG(x,y)|\leq C(\O,\gamma,L,\mu, k,l,\omega)\left(\frac{\max\{|x|,|y|\}}{\min\{|x|,|y|\}}\right)^\gamma |x-y|^{2k-n-l}$$
\item If $|y|<|x|$ and $l\leq 2k-1$, we have that
\beq\label{est:green:1}|\nabla_y^lG(x,y)|\leq C(\O,\gamma,L,\mu, k,l,\omega)\left(\frac{\max\{|x|,|y|\}}{\min\{|x|,|y|\}}\right)^{\gamma+l} |x-y|^{2k-n-l}
\eeq
\end{itemize}
where $C(\O,\gamma,L,\mu, k,l,\omega)$ depends only on $\Omega$, $\gamma$, $L$, $\mu$, $k,l$ and $\omega$. \end{theorem}
\subsection{Construction of the Green's function}

\noindent{\bf Preliminary notations:} In addition to the Sobolev inequality \eqref{sobo:ineq:rn}, we will make a regular use of the Hardy inequality on $\rn$ (see Theorem 3.3 in Mitidieri \cite{mitidieri}): there exists $C_H(n,k)>0$ such that
\begin{equation}\label{APP:GREEN:hardy:ineq:intro}
\int_{\rn}\frac{\varphi^2}{|X|^{2k}}\, dX\leq C_H(n,k)\int_{\rn}(\Delta^{\frac{k}{2}}\varphi)^2\, dX\hbox{ for all }\varphi\in \dkdeux.
\end{equation}
For $\mu>0$, we define
\begin{equation*}
\Pl:=\left\{\begin{array}{cc}
V\in L^1(\Omega)\hbox{ such that}\\
|V(x)|\leq \mu |x|^{-2k}\hbox{ for all }x\in \Mmx\end{array}\right\}.
\end{equation*}
In the sequel,  we consider an operator $P=\Delta^k+h$, where $h\in L^\infty(\O)$ is such that
\begin{equation}\label{APP:GREEN:def:okl}
\Vert h\Vert_{L^\infty}\leq L\hbox{ and } \int_\Omega uPu\,dx\geq L^{-1}\Vert u\Vert_{H_k^2}^2\hbox{ for all }u\in\hundeux.
\end{equation}
\medskip\noindent{\bf Step 0: Approximation of the potential.} 
We claim that there exists $\mu_0=\mu_0(k, L)$ such for all $V_0\in \Pl $ with $0<\mu<\mu_0$, then 
\begin{equation}\label{APP:GREEN:coer:PV:0}
\int_\Omega (Pu-V_0 u)u\, dx\geq \frac{1}{2L}\Vert u\Vert_{H_k^2}^2\hbox{ for all }u\in \hundeux
\end{equation}
and there exists a family $(V_\eps)_{\eps>0} \in L^\infty(\Omega)$ such that:
\begin{equation}\label{APP:GREEN:hyp:Ve}
\left\{\begin{array}{cc}
\lim_{\eps\to 0}V_\eps(x)=V_0(x)&\hbox{ for a.e. }x\in \Mmx\\
V_\eps\in \Pl&\hbox{ for all }\eps>0\\
P-V_\eps&\hbox{ is uniformly coercive for all }\eps>0
\end{array}\right\}
\end{equation}
in the sense that
\begin{equation}\label{APP:GREEN:coer:PV}
\int_M (Pu-V_\eps u)u\, \geq \frac{1}{2L}\Vert u\Vert_{H_k^2}^2\hbox{ for all }u\in \hundeux\hbox{ and }\eps>0
\end{equation}

\smallskip\noindent We prove the claim. The coercivity of $P$ and the Hardy inequality \eqref{APP:GREEN:hardy:ineq:intro} yield
$$\int_\Omega u(P-V_0)u\, dx\geq \frac{1}{L}\Vert u\Vert_{H_k^2}^2-\mu \int_\Omega \frac{u^2}{|x|^{2k}}\, dx\geq \left(\frac{1}{L} -\mu C_H(n,k)\right)\Vert u\Vert_{H_k^2}^2$$
for all $u\in \hundeux$. For $\eta\in C^\infty(\rr)$ such that $\eta(t)=0$ for $t\leq 1$ and $\eta(t)=1$ for $t\geq 2$, define $V_\eps(x):=\eta(|x|/\eps) V_0(x)$ for all $\eps>0$ and a.e. $x\in \Omega$. As one checks, the claim holds with $0<\mu_0<(2C_H(n,k)L)^{-1}$. This proves the claim.

\medskip\noindent For any $\eps>0$,  we let $G_\eps$ be the Green's function for the operator $P-V_\eps$. Since $V_\eps\in L^\infty(\Omega)$, the existence of $G_\eps$ follows from Theorem \ref{APP:GREEN:th:green:nonsing} of the Appendix \ref{APP:GREEN:app:green}.

\medskip\noindent{\bf Step 1: Integral bounds.} We choose $f\in C^{0}_c(\O)$ and we fix $\eps>0$. Since $P-V_\eps$ is coercive, it follows from variational methods that there exists a unique function $\varphi_\eps\in \hundeux$ such that 
\begin{equation*}
\left\{\begin{array}{cc}
(P-V_\eps)\varphi_\eps=f&\hbox{ in }\O\, \\
{\partial_{\nu}^i\varphi_\eps}_{|\partial\O}=0&\hbox{ for }i=0,...,k-1\\
\end{array}\right\}\hbox{ in the weak sense.}
\end{equation*}
It follows from Theorem \ref{APP:GREEN:th:3} and Sobolev's embedding theorem that $\varphi_\eps\in C^{2k-1}(\Obar)$ and $\varphi_\eps\in H_{2k}^p(\O)$ for all $p>1$. The coercivity hypothesis \eqref{APP:GREEN:coer:PV} yields
\begin{eqnarray*}
\frac{1}{2L}\Vert \varphi_\eps\Vert_{H_k^2}^2\leq \int_\O (P\varphi_\eps-V_\eps\varphi_\eps)\varphi_\eps\, dx=\int_\O f\varphi_\eps\, dx\leq \Vert f\Vert_{\frac{2n}{n+2k}}\Vert\varphi_\eps\Vert_{\frac{2n}{n-2k}}
\end{eqnarray*}
With inequality \eqref{sobo:ineq:rn}, we get that
\begin{equation}\label{APP:GREEN:eq:14}
K(n,k)^{-1}\Vert\varphi_\eps\Vert_{\crit}\leq \Vert\varphi_\eps \Vert_{H_k^2}\leq 2LK(n,k)\Vert f\Vert_{\frac{2n}{n+2k}}
\end{equation}
for all $f\in C^{0}_c(\Omega)$. We fix $p>1$ such that
$$\frac{n}{2k}<p<\frac{n}{2k-1}\hbox{ and }\theta_p:=2k-\frac{n}{p}\in (0,1).$$
We fix $\delta\in (0,d(0,\partial\O)/4)$. Since $V_\eps\in \Pl$ for all $\eps>0$ and $P$ satisfies \eqref{APP:GREEN:def:okl}, it follows from regularity theory, see Theorem \ref{APP:GREEN:th:2:again} of Appendix \ref{APP:GREEN:sec:regul:adn}) and Sobolev's embedding theorem that
\begin{eqnarray*}
\Vert\varphi_\eps\Vert_{C^{0,\theta_p}(\Obar\setminus B_\delta(0))}&\leq & C(p,\delta,k)\Vert \varphi_\eps\Vert_{H_{2k}^p(\Omega\setminus B_\delta(0))}\\
&\leq & C(p,\delta,k, L,\mu_0) \left(\Vert f\Vert_{L^p(\Omega\setminus B_{\delta/2}(0))}+\Vert \varphi_\eps\Vert_{L^{\crit}(\Omega\setminus B_{\delta/2}(0))}\right).
\end{eqnarray*}
With \eqref{APP:GREEN:eq:14} and noting that $\frac{n}{2k}>\frac{2n}{n+2k}$, we get that
\begin{equation*}
\Vert\varphi_\eps\Vert_{C^{0,\theta_p}(\Obar\setminus B_\delta(0))}\leq  C(p,\delta,k, L,\mu_0) 
\Vert f\Vert_{L^p(\O)}.
\end{equation*}
Since $\varphi_\eps\in H_{2k}^p(\O)$ for all $p>1$, for any $x\in \Mmx$, Green's representation formula (see Theorem \ref{APP:GREEN:th:green:nonsing}) yields
$$\varphi_\eps(x)=\int_\O G_\eps(x,y)f(y)\, dy\hbox{ for all }x\in \Mmx,$$
and then when $|x|>\delta$, we get
$$\left|\int_\O G_\eps(x,y)f(y)\, dy\right|\leq C(p,\delta,k, L,\mu) 
\Vert f\Vert_{L^p(\O)}$$
for all $f\in C^{0}_c(\O)$ and $p\in \left(\frac{n}{2k},\frac{n}{2k-1}\right)$. Via duality, we then deduce that 
\begin{equation}\label{APP:GREEN:eq:16}
\Vert G_\eps(x,\cdot)\Vert_{L^q(\O)}\leq C(q,\delta, k, L,\mu_0)\hbox{ for all }q\in \left(1,\frac{n}{n-2k}\right)\hbox{ and }|x|>\delta. 
\end{equation}
We now fix $x\in \O$ such that $|x|>\delta$. We take $f\in C^0_c(\O)$ such that $f\equiv 0$ in $B_{\delta/2}(x)$, so that $(P-V_\eps)\varphi_\eps=0$ in $B_{\delta/2}(x)$. Since $V_\eps\in \Pl$ for all $\eps>0$ and $P$ satisfies \eqref{APP:GREEN:def:okl}, it follows from regularity theory (Theorem \ref{APP:GREEN:th:2:again} of Appendix \ref{APP:GREEN:sec:regul:adn}) and Sobolev's embedding theorem that for any $p>\frac{n}{2k}$,
\begin{eqnarray*}
\Vert\varphi_\eps\Vert_{C^{0,\theta_p}(\O\cap B_{\delta/4}(x))}&\leq & C(p,\delta,k)\Vert \varphi_\eps\Vert_{H_{2k}^p(\O\cap B_{\delta/4}(x))}\\
&\leq & C(p,\delta,k, L,\mu_0) \Vert \varphi_\eps\Vert_{L^{\crit}(\O\cap B_{\delta/2}(x))}.
\end{eqnarray*}
With \eqref{APP:GREEN:eq:14}, we get that
\begin{equation*}
\Vert\varphi_\eps\Vert_{C^{0,\theta_p}(  \O\cap B_{\delta/4}(x))}\leq  C(p,\delta,k, L,\mu_0) 
\Vert f\Vert_{L^{\frac{2n}{n+2k}}(\O)}.
\end{equation*}
Since $\varphi_\eps\in H_{2k}^p(\O)$ for all $p>1$ and $\varphi_\eps\in C^{2k-1}(\Obar)\cap H_{k,0}^2(\O)$, Green's representation formula (see Theorem \ref{APP:GREEN:th:green:nonsing}) yields
$$\varphi_\eps(x)=\int_\O G_\eps(x,y)f(y)\, dy,$$
and then 
$$\left|\int_\O G_\eps(x,y)f(y)\, dy\right|\leq C(p,\delta,k, L,\mu) 
\Vert f\Vert_{L^{\frac{2n}{n+2k}}(\O)}$$
for all $f\in C^{0}_c(\O)$ vanishing in $B_{\delta/2}(x)$. Via duality, we then deduce that 
\begin{equation}\label{APP:GREEN:eq:16:crit}
\Vert G_\eps(x,\cdot)\Vert_{L^{\crit}(\O\setminus B_{\delta/2}(x))}\leq C(\delta, k, L,\mu_0)\hbox{ when }|x|>\delta. 
\end{equation}

\medskip\noindent{\bf Step 2: passing to the limit $\eps\to 0$ and Green's function for $P-V_0$.}\par
\noindent We fix $\delta>0$ and $x\in \O$ such that $|x|>\delta$. For all $\eps>0$, we have that 
\begin{equation}\label{APP:GREEN:eq:Ge}
\left\{\begin{array}{cc}
P G_\eps(x,\cdot)-V_\eps G_\eps(x,\cdot)=0&\hbox{ in }\O\setminus\{x\} \\
{\partial_{\nu}^iG_\eps(x,\cdot)}_{|\partial\O}=0&\hbox{ for }i=0,...,k-1\\
\end{array}\right\}
\end{equation}
Since $V_\eps\in \Pl$ for all $\eps>0$, we have that $|V_\eps(y)|\leq C(\mu, \delta)$ for all $y\in \O\setminus B_{\delta/2}(0)$ and $\eps>0$. Since $P$ satisfies \eqref{APP:GREEN:def:okl} and $G_\eps(x,\cdot)\in H_{2k,loc}^p(\O\setminus\{0,x\})\cap H_{k,0,loc}^2(\O\setminus\{x\})$, it follows from the control \eqref{APP:GREEN:eq:16} and standard regularity theory (see Theorem \ref{APP:GREEN:th:2:again}) that given $\nu\in (0,1)$, we have that for any $r>0$,
\begin{equation}\label{APP:GREEN:est:G:1}
\Vert G_\eps(x,\cdot)\Vert_{C^{2k-1,\nu}(\O-(B_r(0)\cup B_r(x))}\leq C(\delta, k, \mu, L,r,\nu,\mu_0)\hbox{ for all }|x|>\delta.
\end{equation}
It then follows from Ascoli's theorem that, up to extraction of a subfamily, there exists $G_0(x,\cdot)\in C^{2k-1}(\Obar-\{x,0\})$ such that
\begin{equation}\label{APP:GREEN:lim:Ge}
\lim_{\eps\to 0}G_\eps(x,\cdot)=G_0(x,\cdot)\hbox{ in }C^{2k-1}_{loc}(\Obar-\{x,0\}).
\end{equation}
By Theorem \ref{APP:GREEN:th:2:again} again, we also get that
\begin{equation}\label{APP:GREEN:lim:Ge:H}
\lim_{\eps\to 0}G_\eps(x,\cdot)=G_0(x,\cdot)\hbox{ in }H^p_{2k,loc}(\Obar-\{x,0\})\hbox{ for all }p>1.
\end{equation}
Moreover, passing to the limit in \eqref{APP:GREEN:eq:16}, we get that 
\begin{equation}\label{APP:GREEN:eq:17}
\Vert G_0(x,\cdot)\Vert_{L^q(\O)}\leq C(q,\delta, k, L,\mu_0)\hbox{ for all }q\in \left(1,\frac{n}{n-2k}\right)\hbox{ and }|x|>\delta,
\end{equation}
and then $G_0(x,\cdot)\in L^q(\O)$ for all $q\in \left(1,\frac{n}{n-2k}\right)$ and $x\neq 0$. Similarly, using \eqref{APP:GREEN:eq:16:crit}, we get that
\begin{equation}\label{APP:GREEN:eq:16:crit:final}
\Vert G_0(x,\cdot)\Vert_{L^{\crit}(\O\setminus B_{\delta/2}(x))}\leq C(\delta, k, L,\mu_0)\hbox{ when }|x|>\delta. 
\end{equation}
So that $G_0(x,\cdot)\in L^{\crit}_{loc}(\O\setminus\{x\})$.

\medskip\noindent{\bf Step 3: Representation formula.} We fix $f\in L^{\frac{2n}{n+2k}}(\O)\cap L^p_{loc}(\Mmx)$, $p>\frac{n}{2k}>1$. Via the coercivity of $P-V_\eps$ and $P-V_0$, it follows from variational methods (see also Theorem \ref{APP:GREEN:th:3}) that there exists  $\varphi_\eps\in \hundeux\cap H_{2k}^{\frac{2n}{n+2k}}(\O)$ and $\varphi_0 \in \hundeux$ such that 
\begin{equation}\label{APP:GREEN:eq:phieps:phi0}
\left\{\begin{array}{cc}
(P-V_\eps)\varphi_\eps=f&\hbox{ in }\O\\
\partial_\nu^i{\varphi_\eps}_{|\partial\O}=0&\hbox{ for }i=0,...,k-1
\end{array}\right\}
\hbox{ and }
\left\{\begin{array}{cc}
(P-V_0)\varphi_0=f&\hbox{ in }\O\\
\partial_\nu^i{\varphi_0}_{|\partial\O}=0&\hbox{ for }i=0,...,k-1
\end{array}\right\}
\end{equation}
As one checks,
\begin{equation}\label{APP:GREEN:lim:phi:eps}
\lim_{\eps\to 0}\varphi_\eps=\varphi_0\hbox{ in }\hundeux\hbox{ and }\lim_{\eps\to 0}\varphi_\eps=\varphi_0\hbox{ in }C^0_{loc}(\Obar\setminus \{0\})
\end{equation}
We now write Green's formula for $\varphi_\eps$ to get
$$\varphi_\eps(x)=\int_\O G_\eps(x,\cdot)f\, dy\hbox{ for }x\neq 0\hbox{ and for all }\eps>0.$$
With \eqref{APP:GREEN:eq:16}, \eqref{APP:GREEN:eq:16:crit},  \eqref{APP:GREEN:lim:Ge}, \eqref{APP:GREEN:eq:17}, \eqref{APP:GREEN:eq:16:crit:final} and \eqref{APP:GREEN:lim:phi:eps}, we pass to the limit  to get
$$\varphi_0(x)=\int_\O G_0(x,\cdot)f\, dy.$$
This yields the existence of a Green's function for $P-V_0$ in Theorem \ref{APP:GREEN:th:Green:main}. Concerning uniqueness, let us consider another Green's function as in Theorem \ref{APP:GREEN:th:Green:main}, say $\bar{G}_0$, and, given $x\in \Mmx$, let us define $H_x:=G_0(x,\cdot)-\bar{G}_0(x,\cdot)$. We then get that $H_x\in L^q(\O)$ for all $1\leq q<\frac{n}{n-2k}$ and $\int_\O H_x f\, dy=0$ for all $f\in C^0_c(\O)$. By density, this identity is also valid for all $f\in L^{q'}(\O)$ where $\frac{1}{q}+\frac{1}{q'}=1$. By duality, this yields $H_x\equiv 0$, and then $\bar{G}_0=G_0$, which proves uniqueness. This ends the proof of Theorem \ref{APP:GREEN:th:Green:main}.

\medskip\noindent{\bf Step 4: First pointwise control.} As above, we fix $\delta>0$ and we take $x\in \Omega$ such that $|x|>\delta$. It follows from \eqref{APP:GREEN:eq:Ge}, \eqref{APP:GREEN:est:G:1}, \eqref{APP:GREEN:lim:Ge} and regularity theory (see Theorem \ref{APP:GREEN:th:2:again}) that for all $l\in \{0,...,2k-1\}$, we have that
\begin{equation}\label{APP:GREEN:est:G:2}
| \nabla_y^lG_\eps(x,y)|\leq  C(\O,\delta, k, \mu_0, L)\hbox{ for  }\{|x-y|\geq\delta,\, |x|\geq\delta,\,|y|\geq\delta\}.
\end{equation}
and
\begin{equation}\label{APP:GREEN:est:G:3}
| \nabla_y^lG_0(x,y)|\leq  C(\O,\delta, k, \mu_0, L)\hbox{ for  }\{|x-y|\geq\delta,\, |x|\geq\delta,\,|y|\geq\delta\}.
\end{equation}
We fix $\gamma\in (0,n-2k)$. Since $G_\eps(x,\cdot)$ satisfies \eqref{APP:GREEN:eq:Ge} in the weak sense and $G_\eps(x,\cdot)\in H_{k}^2(B_{\delta/2}(0))$, it follows from Lemma \ref{APP:GREEN:lem:main} that for all $p>1$, there exists $\mu=\mu(\gamma,L,\delta)>0$, there exists $C=C(\O,\gamma,p, L,\delta)>0$ such that
\begin{equation*}
|y|^{\gamma}|G_\eps(x,y)|\leq C  \Vert G_\eps(x,\cdot)\Vert_{L^p(B_{\delta/2}(0))}\hbox{ for all }y\in B_{\delta/3}(0)-\{0\}
\end{equation*}
when $|x|\geq \delta$. It then follows from \eqref{APP:GREEN:eq:16} that
\begin{equation}\label{APP:GREEN:ineq:36}
|y|^{\gamma}|G_\eps(x,y)|\leq C(\O,\delta,k,L,\gamma) \hbox{ for all }y\in B_{\delta/2}(0)-\{0\}\hbox{ and }|x|\geq \delta.
\end{equation}
With Lemma \ref{APP:GREEN:lem:main}, for all $0\leq l\leq 2k-1$, there exists  $C(\delta,k,L,\gamma,l)>0$ such that
\begin{equation}\label{APP:GREEN:ineq:36:bis}
|y|^{\gamma+l}|\nabla_y^lG_\eps(x,y)|\leq C(\O,\delta,k,L,\gamma,l) \hbox{ for all }y\in B_{\delta/2}(0)-\{0\}\hbox{ and }|x|\geq \delta.
\end{equation}

These inequalities are valid for $\eps>0$, and then for $\eps=0$. In order to get the full estimates of Theorem \ref{APP:GREEN:th:Green:pointwise:BIS}, we now perform infinitesimal versions of these estimates.

\subsection{Asymptotics for the Green's function close to the singularity}\label{APP:GREEN:sec:green2}
We prove an infinitesimal version of \eqref{APP:GREEN:est:G:2} and \eqref{APP:GREEN:ineq:36} for $x,y$ close to the singularity $0$. 

\begin{theorem}\label{APP:GREEN:th:G:close} Let $\Omega$ be a smooth domain of $\rn$ such that $0\in\Omega$ is an interior point. Fix $k\in\nn$ such that $2\leq 2k<n$, $L>0$ and $\mu>0$. Fix an operator $P$ that satisfies \eqref{APP:GREEN:def:okl}, $V\in \Pl$ and a family $(V_\eps)$ as in \eqref{APP:GREEN:hyp:Ve}. For $\mu>0$ sufficiently small, let $G_\eps$ be the Green's function for $P-V_\eps$, $\eps\geq 0$. Let us fix $\mU,\mV$ two open subsets of $\rn$ such that
\begin{equation*}
\mU\subset\subset \rn-\{0\}\, ,\, \mV\subset\subset \rn \hbox{ and }\overline{\mU}\cap\overline{\mV}=\emptyset.
\end{equation*}
We let $\alpha_0:=\alpha_0(\mU,\mV)>0$ be such that $|\alpha X|<d(0,\partial\O)/2$ for all $0<\alpha<\alpha_0$ and $X\in \mU\cup\mV$.  We fix $\gamma\in (0, n-2k)$. Then there exists $\mu=\mu(\gamma)>0$, there exists $C(\mU,\mV,\mu, k, L)>0$ such that 
\begin{equation}\label{APP:GREEN:ineq:G:close}
\left| |X|^\gamma \alpha^{n-2k+l}\nabla_y^lG_\eps( \alpha X,\alpha Y )\right|\leq C(\mU,\mV,k,L)
\end{equation}
for all $X\in \mV-\{0\}$, $Y\in\mU$, $l=0,...,2k-1$, $\alpha\in (0,\alpha_0)$ and $\eps\geq 0$.
\end{theorem}

\noindent{\it Proof of Theorem \ref{APP:GREEN:th:G:close}.} We first set $\mU',\mV'$ two open subsets of $\rn$ such that
\begin{equation*}
\mU\subset\subset \mU'\subset\subset \rn-\{0\}\, ,\, \mV\subset\subset \mV'\subset\subset \rn \hbox{ and }\overline{\mU'}\cap\overline{\mV}=\emptyset.
\end{equation*}
We fix $f\in C^\infty_c(\mU')$ and for any $0<\alpha<\alpha_0$, we set 
$$f_\alpha(x):=\frac{1}{\alpha^{\frac{n+2k}{2}}}f\left(\frac{x}{\alpha}\right)\hbox{ for all }x\in \Omega.$$
As one checks, $f_\alpha\in C^\infty_c( \alpha \mU')$ and $ \alpha \mU'\subset\subset \Mmx$. It follows from Theorem \ref{APP:GREEN:th:3} that there exists $\varphi_{\alpha,\eps}\in H_{2k}^q(\O)\cap H_{k,0}^q(\O)$ for all $q>1$ be such that
\begin{equation}\label{APP:GREEN:eq:phi:alpha}
\left\{\begin{array}{cc}
P\varphi_{\alpha,\eps}-V_\eps\varphi_{\alpha,\eps}=f_\alpha&\hbox{ in }\O\, \\
{\partial_{\nu}^i\varphi_\eps}_{|\partial\O}=0&\hbox{ for }i=0,...,k-1\\
\end{array}\right\}\hbox{ in the weak sense.}
\end{equation}

It follows from Sobolev's embedding theorem that $\varphi_{\alpha,\eps}\in C^{2k-1}(\Obar)$. We define
\begin{equation}\label{APP:GREEN:def:tilde:phi}
\tilde{\varphi}_{\alpha,\eps}(X):= \alpha^{\frac{n-2k}{2}}\varphi_{\alpha,\eps}\left(\alpha X\right)\hbox{ for all }X\in \rn-\{0\},\, |\alpha X|<d(0,\partial\O).
\end{equation}
A change of variable yields
\begin{eqnarray*}
\Vert f_\alpha\Vert_{\frac{2n}{n+2k}}^{\frac{2n}{n+2k}}&=&\int_\O |f_\alpha(x)|^{\frac{2n}{n+2k}}\, dx=\int_{\alpha \mU'} |f_\alpha(x)|^{\frac{2n}{n+2k}}\, dx=  \int_{\mU'} |f(X)|^{\frac{2n}{n+2k}}\, dX.
\end{eqnarray*}
Therefore
\begin{equation}\label{APP:GREEN:eq:31}
\Vert f_\alpha\Vert_{L^{\frac{2n}{n+2k}}(\O)}=\Vert f\Vert_{L^{\frac{2n}{n+2k}}(\mU')}.
\end{equation}
With \eqref{APP:GREEN:coer:PV}, \eqref{APP:GREEN:eq:phi:alpha} and the Sobolev inequality \eqref{sobo:ineq:rn}, we get
\begin{eqnarray*}
\frac{1}{2L}\Vert \varphi_{\alpha,\eps}\Vert_{\hundeux}^2&\leq& \int_\O \varphi_{\alpha,\eps}(P-V_\eps)\varphi_{\alpha,\eps}\, dx=\int_\O f_\alpha\varphi_{\alpha,\eps} \, dx\\
&\leq & \Vert f_\alpha\Vert_{L^{\frac{2n}{n+2k}}(\O)} \Vert \varphi_{\alpha,\eps}\Vert_{L^{\frac{2n}{n-2k}}(\O)}\leq \sqrt{K(n,k)} \Vert f_\alpha\Vert_{L^{\frac{2n}{n+2k}}(\O)} \Vert \varphi_{\alpha,\eps}\Vert_{\hundeux}.
\end{eqnarray*}
Therefore, using again the Sobolev inequality \eqref{sobo:ineq:rn} and \eqref{APP:GREEN:eq:31}, we get that
\begin{equation}\label{APP:GREEN:eq:33}
\Vert \varphi_{\alpha,\eps}\Vert_{L^{\frac{2n}{n-2k}}(\O)}\leq C(n,k,L) \Vert f\Vert_{L^{\frac{2n}{n+2k}}(\mU')} 
\end{equation}
Equation \eqref{APP:GREEN:eq:phi:alpha} rewrites
\begin{equation}\label{APP:GREEN:eq:tpl}
\Delta^k\tilde{\varphi}_{\alpha,\eps}+ \alpha^{2k}h(\alpha \cdot)\tilde{\varphi}_{\alpha,\eps}-\alpha^{2k}V_\eps(\alpha X)\tilde{\varphi}_{\alpha,\eps}=f
\end{equation}
weakly locally in $\rn$. Since $V_\eps$ satisfies \eqref{APP:GREEN:hyp:Ve}, we have that
\begin{equation*}
\left|\alpha^{2k}V_\eps(\alpha X)\right|\leq \mu |X|^{-2k}\hbox{ for all }X\in \mV'-\{0\}
\end{equation*}
Since $f(X)=0$ for all $X\in \mV'$ and $\tilde{\varphi}_{\alpha,\eps}\in H_{2k,loc}^q(\mV')$, it follows from the regularity Lemma \ref{APP:GREEN:lem:main} that there exists $\mu=\mu(\gamma)>0$ such that for any  $\delta>0$ such that $B_{\delta}(0)\subset\subset \mV'$, there exists $C( L, \delta,\gamma, \mV')>0$ such that
\begin{equation*}
|X|^\gamma \left|\tilde{\varphi}_{\alpha,\eps}(X)\right|\leq 
C( L, \delta,\gamma, U,U')\Vert \tilde{\varphi}_{\alpha,\eps} \Vert_{L^{\crit}(\mV')}\hbox{ for all }X\in B_{\delta}(0)-\{0\}
\end{equation*}
Since the coefficients are uniformly bounded outside $0$, classical elliptic regularity yields
\begin{equation*}
 \left|\tilde{\varphi}_{\alpha,\eps}(X)\right|\leq C( L, \delta,\gamma,\mV,\mV')\Vert \tilde{\varphi}_{\alpha,\eps} \Vert_{L^{\crit}(\mV')}\hbox{ for all }X\in \mV- B_{\delta}(0)
\end{equation*}
These two inequalities yield the existence of $C( L, \delta,\gamma,\mV,\mV')$ such that
\begin{equation}\label{APP:GREEN:eq:44}
|X|^\gamma \left|\tilde{\varphi}_{\alpha,\eps}(X)\right|\leq C\Vert \tilde{\varphi}_{\alpha,\eps} \Vert_{L^{\crit}(\mV')}\hbox{ for all }X\in \mV-\{0\}
\end{equation}
Arguing as in the proof of \eqref{APP:GREEN:eq:31}, we have that
\begin{equation}\label{APP:GREEN:eq:32}
\Vert \tilde{\varphi}_{\alpha,\eps}\Vert_{L^{\frac{2n}{n-2k}}(\mV')}\leq \Vert  \varphi_{\alpha,\eps}\Vert_{L^{\frac{2n}{n-2k}}(\O)}.
\end{equation}
Putting together \eqref{APP:GREEN:def:tilde:phi}, \eqref{APP:GREEN:eq:33}, \eqref{APP:GREEN:eq:44} and \eqref{APP:GREEN:eq:32} we get that
\begin{equation}\label{APP:GREEN:ineq:G:a}
|X|^\gamma \left| \alpha^{\frac{n-2k}{2}}\varphi_{\alpha,\eps}\left(\alpha X\right) \right|\leq C( L, \delta, \mu,\gamma,\mV,\mV')\Vert f\Vert_{L^{\frac{2n}{n+2k}}(\mU')}
\end{equation}
for all $X\in \mV-\{0\}$. For $\alpha>0$, we define
\begin{equation}\label{APP:GREEN:def:G:t:close}
\tilde{G}_{\alpha,\eps}(X,Y):=\alpha^{n-2k}G_\eps(\alpha X,\alpha Y )\hbox{ for }(X,Y)\in \mV'\times \mU',\, X\neq 0
\end{equation}
It follows from Green's representation formula for $G_\eps$, $\eps>0$, and \eqref{APP:GREEN:eq:phi:alpha} that
$$\varphi_{\alpha,\eps}\left(\alpha X\right)=\int_\O G_\eps\left(\alpha X,y\right)f_\alpha(y)\, dy$$
for all $X\in \mV-\{0\}$. 
With a change of variable, we then get that
\begin{equation}\label{APP:GREEN:ineq:G:b}
\alpha^{\frac{n-2k}{2}}\varphi_{\alpha,\eps}\left(\alpha X\right)=\int_{\mU'} \tilde{G}_{\alpha,\eps}(X,Y) f(Y)\, dY
\end{equation}
for all $X\in \mV-\{0\}$. Putting together \eqref{APP:GREEN:ineq:G:a} and \eqref{APP:GREEN:ineq:G:b}, we get that
\begin{equation*}
\left||X|^\gamma\int_{\mU'} \tilde{G}_{\alpha,\eps}(X,Y) f(Y)\, dY
\right|\leq C( L, \delta, \mu,\gamma,\mV,\mV',\omega') \Vert f\Vert_{L^{\frac{2n}{n+2k}}(\mU')}\end{equation*}
for all $f\in C^\infty_c(\mU')$ and $ X\in \mV-\{0\}$. It then follows from duality arguments that
\begin{equation}\label{APP:GREEN:ineq:35}
\Vert |X|^\gamma \tilde{G}_{\alpha,\eps}(X,\cdot)\Vert_{L^{\crit}(\mU')}\leq C( L, \delta, \mu,\gamma,\mV,\mV',\mU')\hbox{ for }X\in \mV-\{0\} 
\end{equation}
Since $G_\eps(x,\cdot)$ is a solution to $(P-V_\eps)G_\eps(x,\cdot)=0$ in $\O-\{0,x\}$, as in \eqref{APP:GREEN:eq:tpl},
we get that
\begin{eqnarray*}
&&\Delta ^k\tilde{G}_{\alpha,\eps}(X,\cdot)+\alpha^{2k}h(\alpha \cdot)\tilde{G}_{\alpha,\eps}(X,\cdot)\\
&&-\alpha^{2k}V_\eps(\alpha \cdot)\tilde{G}_{\alpha,\eps}(X,\cdot)=0\hbox{ in }\mU'\subset\subset \rn-\{0,X\} 
\end{eqnarray*}
Since $\mU'\subset\subset \rn-\{0,X\}$, there exists $c_{\mU'}>0$ such that $|Y|\geq c_{\mU'}$ for all $Y\in\mU'$. Since $V_\eps$ satisfies \eqref{APP:GREEN:hyp:Ve}, we have that
\begin{equation*}
\left|\alpha^{2k}V_\eps(\alpha Y)\right|\leq \mu c_{\omega'}^{-2k}\hbox{ for all }Y\in \mU'.
\end{equation*}
It then follows from elliptic regularity theory (see Theorem \ref{APP:GREEN:th:2:again}) that
\begin{equation*}
|X|^\gamma |\tilde{G}_{\alpha,\eps}(X,Y)|\leq C(k,L, \mu,\mU',\mV')\Vert |X|^\gamma \tilde{G}_{\alpha,\eps}(X,\cdot)\Vert_{L^{\crit}(\mU')}
\end{equation*}
for all $Y\in \mU\subset\subset\mU'$ and $X\in \mV-\{0\}$. The conclusion \eqref{APP:GREEN:ineq:G:close}  of Theorem \ref{APP:GREEN:th:G:close} then follows from this inequality, \eqref{APP:GREEN:ineq:35}, the definition \eqref{APP:GREEN:def:G:t:close} of $\tilde{G}_{\alpha,\eps}$, the limit \eqref{APP:GREEN:lim:Ge} and elliptic regularity for the derivatives along $y$.

\subsection{Asymptotics for the Green's function far from the singularity}\label{APP:GREEN:sec:green3}
We prove an infinitesimal version of \eqref{APP:GREEN:est:G:2} and \eqref{APP:GREEN:ineq:36} for $x,y$ far from the singularity. 

\begin{theorem}\label{APP:GREEN:th:G:far} We fix $p\in \Mmx$ and $\mU,\mV$ two open subsets of $\rn$ such that
\begin{equation*}
\mU\subset\subset \rn \, ,\, \mV\subset\subset \rn \hbox{ and }\overline{\mU}\cap\overline{\mV}=\emptyset.
\end{equation*}
We let $\alpha_0>0$ be such that 
\begin{equation}\label{APP:GREEN:hyp:alpha}
|\alpha X|<\frac{1}{2}\min\{d(0,\partial\O),|p|, d(p,\partial\O)\} \hbox{ for all }0<\alpha<\alpha_0\hbox{ and }X\in \mV\cup \mU.
\end{equation}
Then for all $\gamma\in (0,n-2k)$, there exists $\mu=\mu(\gamma)>0$ and $C(\mV,\mU, L,\alpha_0,\gamma,\mu)>0$ such that 
\begin{equation}\label{APP:GREEN:ineq:G:far}
\left|  \alpha^{n-2k+l}\nabla_y^lG_\eps(p+\alpha X,p+\alpha Y )\right|\leq C(U,\omega, L,\alpha_0,\gamma,\mu)
\end{equation}
for all $X\in \mV $ and $Y\in\mU$, $l=0,...,2k-1$, $\alpha\in (0,\alpha_0)$ and $\eps\geq 0$ small enough.
\end{theorem}

\noindent{\it Proof of Theorem \ref{APP:GREEN:th:G:far}.} We first set $\mU',\mV'$ two open subsets of $\rn$ such that
\begin{equation*}
\mU\subset\subset \mU'\subset\subset \rn \, ,\, \mV\subset\subset \mV'\subset\subset \rn \hbox{ and }\overline{\mU'}\cap\overline{\mV'}=\emptyset
\end{equation*}
and \eqref{APP:GREEN:hyp:alpha} still holds for $X\in \mV'\cup\mU'$. We fix $f\in C^\infty_c(\mU')$ and for any $0<\alpha<\alpha_0$, we set 
$$f_\alpha(x):=\frac{1}{\alpha^{\frac{n+2k}{2}}}f\left(\frac{x-p}{\alpha}\right)\hbox{ for all }x\in \O.$$
As one checks, $f_\alpha\in C^\infty_c(p+\alpha \mU')$ and $p+\alpha \mU'\subset\subset \Mmx$. It follows from Theorem \ref{APP:GREEN:th:3} that there exists $\varphi_{\alpha,\eps}\in H_{2k}^q(\Omega)\cap H_{k,0}^q(\Omega)$ for all $q>1$ such that
\begin{equation}\label{APP:GREEN:eq:phi:alpha:far}
\left\{\begin{array}{cc}
P\varphi_{\alpha,\eps}-V_\eps\varphi_{\alpha,\eps}=f_\alpha&\hbox{ in }\O\\
{\partial_{\nu}^i\varphi_{\alpha,\eps}}_{|\partial\O}=0&\hbox{ for }i=0,...,k-1.
\end{array}\right\}
\end{equation}
It follows from Sobolev's embedding theorem that $\varphi_{\alpha,\eps}\in C^{2k-1}(\Obar)$. We define
\begin{equation}\label{APP:GREEN:def:tilde:phi:far}
\tilde{\varphi}_{\alpha,\eps}(X):= \alpha^{\frac{n-2k}{2}}\varphi_{\alpha,\eps}\left(p+\alpha X\right)\hbox{ for all }X\in \rn,\, |\alpha X|<d(p,\partial\O).
\end{equation}
As in \eqref{APP:GREEN:eq:31} and  \eqref{APP:GREEN:eq:33}, we get
\begin{equation}\label{APP:GREEN:eq:31:far}
\Vert f_\alpha\Vert_{L^{\frac{2n}{n+2k}}(\O)}=\Vert f\Vert_{L^{\frac{2n}{n+2k}}(\mU')}\hbox{ and }\Vert \varphi_{\alpha,\eps}\Vert_{L^{\frac{2n}{n-2k}}(\O)}\leq C(n,k,L) \Vert f\Vert_{L^{\frac{2n}{n+2k}}(\mU')}.
\end{equation}
Equation \eqref{APP:GREEN:eq:phi:alpha:far} rewrites
\begin{equation}\label{APP:GREEN:eq:tpl:far}
\Delta^k\tilde{\varphi}_{\alpha,\eps}+\alpha^{2k}h(p+\alpha \cdot)\tilde{\varphi}_{\alpha,\eps}-\alpha^{2k}V_\eps(p+\alpha X)\tilde{\varphi}_{\alpha,\eps}=f
\end{equation}
weakly in $\rn$. Since $V_\eps$ satisfies \eqref{APP:GREEN:hyp:Ve}, we have that
\begin{equation*}
\left|\alpha^{2k}V_\eps(p+\alpha X)\right|\leq \mu \alpha^{2k}|p+\alpha X|^{-2k}\hbox{ for all }X\in \mV' 
\end{equation*}
With \eqref{APP:GREEN:hyp:alpha}, we get that
\begin{equation*}
\left|\alpha^{2k}V_\eps(p+\alpha X)\right|\leq \mu\left(\frac{2|p|}{\alpha}\right)^{-2k} \leq C(\mu_0)\hbox{ for all }X\in \mV' 
\end{equation*}
Since $f(X)=0$ for all $X\in \mV'$, it follows from standard regularity theory (see Theorem \ref{APP:GREEN:th:2:again}) that there exists $C( k,L, \mV, \mV',\mU,\mU',\alpha_0)>0$ such that
\begin{equation}\label{APP:GREEN:eq:64}
  \left|\tilde{\varphi}_{\alpha,\eps}(X)\right|\leq C\Vert \tilde{\varphi}_{\alpha,\eps} \Vert_{L^{\crit}(\mV')}\hbox{ for all }X\in \mV
\end{equation}
Arguing as in the proof of \eqref{APP:GREEN:eq:31}, we have that
\begin{equation}\label{APP:GREEN:eq:32:far}
\Vert \tilde{\varphi}_{\alpha,\eps}\Vert_{L^{\frac{2n}{n-2k}}(\mV')}\leq \Vert  \varphi_{\alpha,\eps}\Vert_{L^{\frac{2n}{n-2k}}(\O)}.
\end{equation}
Putting together \eqref{APP:GREEN:def:tilde:phi:far}, \eqref{APP:GREEN:eq:64}, \eqref{APP:GREEN:eq:32:far} and \eqref{APP:GREEN:eq:31:far} we get that
\begin{equation}\label{APP:GREEN:ineq:G:a:far}
  \left| \alpha^{\frac{n-2k}{2}}\varphi_{\alpha,\eps}\left(p+\alpha X\right) \right|\leq C( k,L, \mV, \mV',\mu) \Vert f\Vert_{L^{\frac{2n}{n+2k}}(\mU')} \hbox{ for all }X\in \mV.
\end{equation}
We now just follow verbatim the proof of Theorem \ref{APP:GREEN:th:G:close} above to get the conclusion \eqref{APP:GREEN:ineq:G:far}  of Theorem \ref{APP:GREEN:th:G:far}. We leave the details to the reader.

\subsection{Proof of Theorem \ref{APP:GREEN:th:Green:pointwise:BIS}}\label{APP:GREEN:sec:green4}
We let $\O$, $k$, $\mu$, $L$, $P$, $V$ as in the statement of Theorem \ref{APP:GREEN:th:Green:pointwise:BIS}. With $\mu>0$ small enough, we let $G_0$ be the Green's function of $P-V$ as in Theorem \ref{APP:GREEN:th:Green:main}. Given $\gamma\in (0,n-2k)$, we let $\mu_\gamma>0$ as in \eqref{APP:GREEN:ineq:36} and Theorems \ref{APP:GREEN:th:G:close} and \ref{APP:GREEN:th:G:far} hold when $0<\mu<\mu_\gamma$. We prove here the first estimate of Theorem \ref{APP:GREEN:th:Green:pointwise:BIS} by contradiction. We fix $\omega\subset\subset\O$ 
and we assume that there is a family of operators $(P_i)_{i\in\nn}$ such that $P_i$ satisfies \eqref{APP:GREEN:def:okl} for all $i$, a family of potentials $(V_i)_{i\in\nn}\in {\mathcal P}_{\mu_\gamma}$, sequences $(x_i), (y_i)\in \Mmx$ such that $x_i\neq y_i$ and $x_i\in\omega$ for all $i\in\nn$ and
\begin{equation}\label{APP:GREEN:hyp:absurd}
\lim_{i\to +\infty}\frac{|x_i-y_i|^{n-2k}|G_i(x_i,y_i)|}{\left(\frac{\max\{|x_i|,|y_i|\}}{\min\{|x_i|,|y_i|\}}\right)^\gamma }=+\infty,
\end{equation}
where $G_i$ denotes the Green's function of $P_i-V_i$ for all $i\in\nn$. We distinguish 5 cases:

\smallskip\noindent{\it Case 1:} $|x_i-y_i|=o(|x_i|)$ as $i\to +\infty$. It then follows from the triangle inequality that $|x_i-y_i|=o(|y_i|)$ and $|x_i|=(1+o(1))|y_i|$. Therefore
$$\left(\frac{\max\{|x_i|,|y_i|\}}{\min\{|x_i|,|y_i|\}}\right)^\gamma =1+o(1)$$
and then \eqref{APP:GREEN:hyp:absurd} yields
\begin{equation}\label{APP:GREEN:hyp:absurd:1}
\lim_{i\to +\infty} |x_i-y_i|^{n-2k}|G_i(x_i,y_i)|=+\infty
\end{equation}
We let $Y_i\in\rn$ be such that $y_i:=x_i+|x_i-y_i|Y_i$. In particular, $|Y_i|=1$, so, up to a subsequence, there exists $Y_\infty\in\rn$ such that $\lim_{n\to +\infty}Y_i=Y_\infty$ with $|Y_\infty|=1$
Note that since $x_i\in\omega$, there exists $\eps_0>0$ such that $d(x_i,\partial\O)\geq \eps_0$ for all $i$. We apply Theorem \ref{APP:GREEN:th:G:far} with $p:=x_i$, $\alpha:=|x_i-y_i|$, $\mV=B_{1/3}(0)$, $\mU=B_{1/3}(Y_\infty)$: for $i\in\nn$ large enough, taking $X=0$ and $Y=Y_i$ in \eqref{APP:GREEN:ineq:G:far}, we get that
$$|x_i-y_i|^{n-2k}|G_i(x_i,y_i )|=|x_i-y_i|^{n-2k}|G_i(x_i+|x_i-y_i|\cdot 0, x_i+|x_i-y_i|\cdot Y_i )|\leq C(L,\gamma,\mu)$$
which contradicts \eqref{APP:GREEN:hyp:absurd:1}. This ends Case 1.

\smallskip\noindent The case  $|x_i-y_i|=o(|y_i|)$ as $i\to +\infty$ is equivalent to Case 1.

\medskip\noindent{\it Case 2:} $|x_i|=o(|x_i-y_i|)$ and $|x_i-y_i|\not\to 0$ as $i\to +\infty$. Therefore  \eqref{APP:GREEN:hyp:absurd} rewrites
\begin{equation}\label{APP:GREEN:hyp:absurd:2bis}
\lim_{i\to +\infty}  |x_i|^{\gamma}|G_i(x_i,y_i)| =+\infty
\end{equation}
This is a contradiction with \eqref{APP:GREEN:ineq:36} when $\eps=0$. This ends Case 2 by using the symmetry of $G$.

\medskip\noindent{\it Case 3:} $|x_i|=o(|x_i-y_i|)$ and $|x_i-y_i|\to 0$ as $i\to +\infty$. Then $|x_i|=o(|y_i|)$ and  $|x_i-y_i|=(1+o(1))|y_i|$. In particular, $x_i,y_i\to 0$ as $i\to+\infty$. Therefore  \eqref{APP:GREEN:hyp:absurd} rewrites
\begin{equation}\label{APP:GREEN:hyp:absurd:2}
\lim_{i\to +\infty}   |x_i-y_i|^{n-2k-\gamma}|x_i|^{\gamma}|G_i(x_i,y_i)| =+\infty
\end{equation}
We let $X_i, Y_i\in\rn$ be such that  $x_i:=|x_i-y_i|X_i$ and  $y_i:=|x_i-y_i|Y_i$. In particular, $\lim_{i\to +\infty}|X_i|=0$ and $|Y_i|=1+o(1)$. So, up to a subsequence, there exists $Y_\infty\in\rn$ such that $\lim_{n\to +\infty}Y_i=Y_\infty$ with $|Y_\infty|=1$. We apply Theorem \ref{APP:GREEN:th:G:close} with  $\alpha:=|x_i-y_i|$, $\mV=B_{1/3}(0)$, $\mU=B_{1/3}(Y_\infty)$: for $i\in\nn$ large enough, taking $X=X_i\neq 0$ and $Y=Y_i$ in \eqref{APP:GREEN:ineq:G:close}, we get that
$$|X_i|^\gamma |x_i-y_i|^{n-2k}|G_i(|x_i-y_i|X_i,|x_i-y_i|Y_i)|\leq C(\mu,k,L),$$
and, coming back to the definitions of $X_i$ and $Y_i$, we get a contradiction with  \eqref{APP:GREEN:hyp:absurd:2}. This ends Case 3.

\medskip\noindent{\it Case 4:} $|y_i|=o(|x_i-y_i|)$ as $i\to +\infty$. Since the Green's function is symmetric, this is similar to Case 2 and 3.

\medskip\noindent{\it Case 5:}  $|x_i|\asymp  |y_i|\asymp  |x_i-y_i| $. Then \eqref{APP:GREEN:hyp:absurd} rewrites
\begin{equation}\label{APP:GREEN:hyp:absurd:3}
\lim_{i\to +\infty} |x_i-y_i|^{n-2k}|G_i(x_i,y_i)|=+\infty
\end{equation}
\smallskip{\it Case 5.1:} $|x_i-y_i|\not\to 0$ as $i\to +\infty$. Then it follows from \eqref{APP:GREEN:est:G:2} that $|G_i(x_i,y_i)|\leq C(\mu,k,L)$ for all $i$, which contradicts \eqref{APP:GREEN:hyp:absurd:3}.\par

\smallskip\noindent {\it Case 5.2:} $|x_i-y_i|\to 0$ as $i\to +\infty$. We let $X_i, Y_i\in\rn$ be such that  $x_i:=|x_i-y_i|X_i$ and  $y_i:=|x_i-y_i|Y_i$. In particular, there exists $c>0$ such that $c^{-1}<|X_i|,|Y_i|<c$ and $|X_i-Y_i|\geq c^{-1}$ for all $i$. So, up to a subsequence, there exists $X_\infty,Y_\infty\in\rn$ such that $\lim_{n\to +\infty}X_i=X_\infty\neq 0$  and $\lim_{n\to +\infty}Y_i=Y_\infty\neq 0$ and $X_\infty\neq Y_\infty$. We apply Theorem \ref{APP:GREEN:th:G:close} with  $\alpha:=\alpha_i=|x_i-y_i|$, $\mV=B_{r_0}(X_\infty)$, $\mU=B_{r_0}(Y_\infty)$ for some $r_0>0$ small enough. So for $i\in\nn$ large enough, taking $X=X_i\neq 0$ and $Y=Y_i$ in \eqref{APP:GREEN:ineq:G:close}, we get that
\begin{equation*}
|X_i|^\gamma \alpha_i^{n-2k}\left|   G_i( \alpha_i X_i,\alpha_i Y_i )\right|\leq C(U,\omega,L,\gamma,\mu)
\end{equation*}
and, coming back to the definitions of $X_i$ and $Y_i$, we get that a contradiction with \eqref{APP:GREEN:hyp:absurd:3}. This ends Case 5.

\medskip\noindent Therefore, in all 5 cases, we have obtained a contradiction with \eqref{APP:GREEN:hyp:absurd}. This proves the first estimate of  Theorem \ref{APP:GREEN:th:Green:pointwise:BIS}. The proof of the estimates on the derivative uses the same method by contradiction, with a few more cases to study using regularity theory (Theorem \ref{APP:GREEN:th:2:again}). We leave the details to the reader.

\section{The regularity Lemma}\label{APP:GREEN:sec:lemma}
For any domain $D\subset\rn$, $k\in\nn$ such that $2\leq 2k<n$ and $L>0$, we say that an operator $P$ is of type $O_{k,L}(D)$ if $P:=\Delta^k+h$, where $h\in L^\infty(D)$ and $\Vert h\Vert_{\infty}\leq L$.

\begin{lemma}\label{APP:GREEN:lem:main} Let $k\in\nn$ be such that $2\leq 2k<n$ and $\delta,L>0$. Fix $p>1$ and $\delta_1,\delta_2>0$ such that $0<\delta_1<\delta_2$. We consider a differential operator $P\in O_{k,L}(B_{\delta_2}(0))$ where $B_{\delta_2}(0)\subset\rn$. Then for all $0<\gamma<n-2k$, there exists $\mu=\mu(\gamma,p, L,\delta_1,\delta_2)> 0$ and $C_0=C_0(\gamma,p, L,\delta_1,\delta_2)>0$ such that for any $V\in L^1(B_{\delta_2}(0))$ such that
$$|V(x)|\leq \mu |x|^{-2k}\hbox{ for all }x\in B_{\delta_2}(0),$$
then for any $\varphi\in H_k^2(B_{\delta_2}(0))\cap H_{2k,loc}^s(B_{\delta_2}(0)-\{0\})$ (for some $s>1$)  such that
$$P\varphi-V\cdot\varphi=0\hbox{ weakly in }H_k^2(B_{\delta_2}(0)),$$
then we have that
\begin{equation}\label{APP:GREEN:ineq:lem}
|x|^{\gamma}|\varphi(x)|\leq C_0\cdot \Vert\varphi\Vert_{L^p(B_{\delta_2}(0))}\hbox{ for all }x\in B_{\delta_1}(0)-\{0\}.
\end{equation}
and
\begin{equation*}
\Vert \varphi\Vert_{H_k^2(B_{\delta_1}(0))}\leq C_0\cdot \Vert\varphi\Vert_{L^p(B_{\delta_2}(0))}.
\end{equation*}
Moreover, for any $0< l<2k$, there exists $C_l=C_l(\gamma,p, L,\delta_1,\delta_2)>0$ such that
\begin{equation}\label{APP:GREEN:ineq:lem:l}
|x|^{\gamma+l}|\nabla^l\varphi(x)|\leq C_l\cdot \Vert\varphi\Vert_{L^p(B_{\delta_2}(0))}\hbox{ for all }x\in B_{\delta_1}(0)-\{0\}
\end{equation}
\end{lemma}
For the reader's convenience, we set $\delta:=\delta_1$ and we assume that $\delta_2=3\delta_1=3\delta$. The general case follows the same proof by changing $2\delta$, $2.9\delta$, etc, into various radii $\delta',\delta^{\prime\prime},...$ such that $\delta_1<\delta'<\delta^{\prime\prime}<\delta_2$, etc. We split the proof of the Lemma in two steps.

\medskip\noindent{\bf Step 1: Proof of \eqref{APP:GREEN:ineq:lem} when $V\equiv 0$ around $0$.} 
We prove \eqref{APP:GREEN:ineq:lem} by contradiction under the assumption that $V$ vanishes around $0$. We assume that there exists $\gamma\in (0,n-2k)$, $p>1$, $L>0$, $\delta>0$ such that for all $\mu>0$, there exists a differential operator $P_\mu=\Delta^k+h_\mu$ and a potential $V_{\mu}\in L^1(B_{3\delta}(0))$ such that there exists $\varphi_{\mu}\in H_k^2(B_{3\delta}(0))\cap   H_{2k,loc}^s(B_{3\delta}(0)-\{0\})$ (for some $s>1$) such that
\begin{equation}\label{APP:GREEN:eq:75}
\left\{\begin{array}{l}
(P_\mu-V_\mu)\psi_\mu=0 \hbox{ weakly in }H_k^2(B_{3\delta}(0))\cap H_{2k,loc}^s(B_{3\delta}(0)-\{0\})\\
\Vert\psi_\mu\Vert_{L^p(B_{3\delta}(0)}=1\\
|V_{\mu}(x)|\leq \mu |x|^{-2k}\hbox{ for all }x\in B_{3\delta}(0)-\{0\}\\
V_{\mu}\equiv 0\hbox{ around }0\\
\sup_{x\in \overline{B_\delta(0)}} |x|^{\gamma}|\psi_\mu(x)|>\frac{1}{\mu}\to +\infty\hbox{ as }\mu\to 0
\end{array}\right\}
\end{equation} 
With our assumption that $V_\mu$ vanishes around $0$, we get that $V_\mu\in L^\infty(B_{3\delta}(0))$. Then, by regularity theory (see Theorem \ref{APP:GREEN:th:2:again}), we get that $\psi_\mu\in C^0(B_{2\delta}(0))$. Therefore, there exists $x_\mu\in \overline{B_\delta(0)}$ such that
\begin{equation}\label{APP:GREEN:lim:xlambda}
|x_\mu|^{\gamma}|\psi_\mu(x_\mu)|=\sup_{x\in \overline{B_\delta(0)}} |x|^{\gamma}|\psi_\mu(x)|>\frac{1}{\mu}\to +\infty
\end{equation}
as $\mu\to 0$. 

\medskip\noindent{\bf Step 1.1:} We claim that $\lim_{\mu\to 0}x_\mu=0$. 

\smallskip\noindent We prove the claim. For any $r>0$, we have that $|V_\mu(x)|\leq \mu r^{-2k}$ for all $x\in B_{3\delta}(0)\setminus B_r(0)$. So, with regularity theory (see Theorem \ref{APP:GREEN:th:2:again}), we get that for all $q>1$, then $\Vert \psi_\mu\Vert_{H_{2k}^q(B_{2\delta}(0)\setminus B_{2r}(0))}= C(r, q, L,p, )\Vert\psi_\mu\Vert_{L^p(B_{3\delta}(0)}\leq C(r,q, L,p)$. Taking $q>\frac{n}{2k}$, we get that $|\psi_\mu(x)|\leq C(r,q, L,p)$ for all $x\in B_{2\delta}(0)\setminus B_{2r}(0)$. With \eqref{APP:GREEN:lim:xlambda}, this forces $\lim_{\mu\to 0}x_\mu=0$. The claim is proved.

\medskip\noindent{\bf Step 1.2: Convergence after rescaling.} We set $r_\mu:=|x_\mu|>0$ and we define
\begin{equation*}
\tilde{\psi}_\mu(X):=\frac{\psi_\mu(r_\mu X)}{\psi_\mu(x_\mu)}\hbox{ for }X\in\rn-\{0\}\hbox{ such that }|X|<\frac{\delta}{r_\mu}.
\end{equation*}
We define $X_\mu\in \rn$ such that $x_\mu=r_\mu X_\mu$. In particular $|X_\mu|=1$. With the definition of $x_\mu$, for any $X\in \rn$ such that $0<|X|<\frac{\delta}{r_\mu}$, we have that
\begin{eqnarray*}
r_\mu^\gamma |X|^\gamma|\tpl(X)|&=& \frac{|r_\mu X|^\gamma |\psi_\mu(r_\mu X)|}{|\psi_\mu(x_\mu)|}\leq \frac{|x_\mu|^\gamma |\psi_\mu(x_\mu)|}{|\psi_\mu(x_\mu)|}=r_\mu^\gamma.
\end{eqnarray*}
Therefore, we get that
\begin{equation}\label{APP:GREEN:bnd:tpl}
|X|^\gamma |\tpl(X)|\leq 1\hbox{ for all }X\in \rn\hbox{ such that }0<|X|<\frac{\delta}{r_\mu}\hbox{ and }\tpl(X_\mu)=1.
\end{equation}
The equation satisfied by $\tpl$ in \eqref{APP:GREEN:eq:75} rewrites
\begin{equation}\label{APP:GREEN:eq:tpl:bis}
\Delta^k\tpl+r_\mu^{2k}h_\mu (r_\mu \cdot)\tpl-r_\mu^{2k}V_\mu(r_\mu X)\tpl=0
\end{equation}
weakly in $B_{3\delta/r_\mu}(0)-\{0\}$. Note that
\begin{equation}\label{APP:GREEN:bnd:Vl}
|r_\mu^{2k}V_\mu(r_\mu X)|\leq \mu |X|^{-2k}\hbox{ for all }\mu>0\hbox{ and }0<|X|<\frac{3\delta}{r_\mu}.
\end{equation}
With the bound \eqref{APP:GREEN:bnd:tpl} and the bounds of the coefficient $h_\mu$, it follows from regularity theory (see Theorem \ref{APP:GREEN:th:2:again}) that for any $R>0$ and any $0<\nu<1$, there exists $C(R)>0$ such that $\Vert \tpl\Vert_{C^{2k-1,\nu}(B_R(0)-B_{R^{-1}}(0))}\leq C(R,\nu)$  for all $\mu>0$. Ascoli's theorem yields the existence of $\tilde{\psi}\in C^{2k-1}(\rn-\{0\})$ such that $\tpl\to \tilde{\psi}$ in $C^{2k-1}_{loc}(\rn-\{0\})$ as $\mu\to 0$. Passing to the limit $\mu\to 0$ in \eqref{APP:GREEN:eq:tpl:bis}, we get that $\Delta^k\tilde{\psi}=0$ weakly in $\rn-\{0\}$ and regularity yields $\tilde{\psi}\in C^{2k}(\rn-\{0\})$. We define $X_0:=\lim_{\mu\to 0}X_\mu$, so that $|X_0|=1$. Finally, passing to the limit in \eqref{APP:GREEN:bnd:tpl} yields
 \begin{equation}\label{APP:GREEN:eq:psi}
\left\{\begin{array}{l}
\tilde{\psi}\in C^{2k}(\rn-\{0\})\\
\Delta^k\tilde{\psi}=0 \hbox{  in }\rn-\{0\}\\
\tilde{\psi}(X_0)=1\hbox{ with }|X_0|=1\\
|\tilde{\psi}(X)|\leq |X|^{-\gamma}\hbox{ for all }X\in \rn-\{0\}.
\end{array}\right\}
\end{equation} 
By standard elliptic theory (see Theorems \ref{APP:GREEN:th:2:again} and \ref{APP:GREEN:th:2:holder}), for any $l=1,...,2k$, there exists $C_l>0$ such that
\begin{equation}\label{APP:GREEN:bnd:nabla:psi}
|\nabla^l\tilde{\psi}(X)|\leq C_l |X|^{-\gamma-l}\hbox{ for all }X\in \rn-\{0\}.
\end{equation}
\noindent{\bf Step 1.3: Contradiction via Green's formula.} Let us consider the Poisson kernel of $\Delta^k$ at $X_0$, namely
$$\Gamma_{X_0}(X):=C_{n,k}|X-X_0|^{2k-n}\hbox{ for all }X\in \rn-\{X_0\},$$
where
\begin{equation}\label{APP:GREEN:def:Cnk}
C_{n,k}:=\frac{1}{(n-2)\omega_{n-1}\Pi_{i=1}^{k-1}(n-2k+2(i-1))(2k-2i)}.
\end{equation}
Let us choose $R>3$ and $0<\eps<1/2$ and define the domain
$$\Omega_{R,\eps}:=B_R(0)\setminus\left(B_{R^{-1}}(0)\cup B_{\eps}(X_0)\right).$$
Note that all the balls involved here have boundaries that do not intersect. With \eqref{formula:ipp}, we get
\begin{equation}\label{APP:GREEN:ipp:gamma}
\int_{\Omega_{R,\eps}}(\Delta^k\Gamma_{X_0})\tilde{\psi}\, dX=\int_{\Omega_{R,\eps}}\Gamma_{X_0}(\Delta^k\tilde{\psi})\, dX+\int_{\partial\Omega_{R,\eps}}\sum_{i=0}^{k-1}\mathcal{B}^{(i)}(\Gamma_{X_0},\tilde{\psi})\, d\sigma
\end{equation}
where the $\mathcal{B}^{(i)}$ are as in \eqref{def:B}. We have that $\partial\Omega_{R,\eps}=\partial B_R(0)\cup \partial B_{R^{-1}}(0)\cup \partial B_{\eps}(X_0)$. Using that $\Gamma_{X_0}$ is smooth at $0$, that $\tilde{\psi}$ is smooth at $X_0$, using the bounds \eqref{APP:GREEN:bnd:nabla:psi} and the corresponding ones for $\Gamma_{X_0}$, for any $i=0,...,k-1$, we get that
\begin{eqnarray*}
&&\left|\int_{\partial B_R(0)}\mathcal{B}^{(i)}(\Gamma_{X_0},\tilde{\psi})\, d\sigma\right| 
\leq CR^{-\gamma}\, ,\, \left|\int_{\partial B_{\eps}(X_0)}\mathcal{B}^{(i)}(\Gamma_{X_0},\tilde{\psi})\, d\sigma\right| \leq  
C\eps^{2i}
\end{eqnarray*}
\begin{eqnarray*}
\left|\int_{\partial B_{R^{-1}}(0)}\mathcal{B}^{(i)}(\Gamma_{X_0},\tilde{\psi})\, d\sigma\right| 
&\leq &CR^{2-n+\gamma+2i}\leq C R^{-(n-2k-\gamma)}
\end{eqnarray*}
and
$$\left|\int_{\partial B_{\eps}(X_0)}\Delta ^{k-1}\Gamma_{X_0}\partial_\nu\tilde{\psi}\, d\sigma\right|\leq C \eps^{n-1}\eps^{2k-n-2(k-1)}\leq C\eps.$$
Therefore, since $0<\gamma<n-2k$, all the terms involving $R$ go to $0$ as $R\to +\infty$, the terms involving $\eps$ go to $0$ when $i\neq 0$. Since $\Delta^k\Gamma_{X_0}=0$, $\Delta^k\tilde{\psi}=0$, it follows from \eqref{APP:GREEN:ipp:gamma} and the inequalities above that
$$\int_{\partial B_{\eps}(X_0)}\partial_\nu\Delta ^{k-1}\Gamma_{X_0} \tilde{\psi}\, d\sigma=o(1)\hbox{ as }\eps\to 0.$$
With the definition of $\Gamma_{X_0}$, we get that 
\begin{equation}\label{APP:GREEN:calc:delta:gamma}
-\partial_\nu\Delta^{k-1}\Gamma_{X_0}(X)=\frac{1}{\omega_{n-1}}|X-X_0|^{1-n}\hbox{ for }X\neq X_0.
\end{equation}
So that, with a change of variable, we get that
$$\int_{\partial B_1(0)} \tilde{\psi}(X_0+\eps X)\, d\sigma=o(1)\hbox{ as }\eps\to 0.$$
Passing to the limit, we get that $\tilde{\psi}(X_0)=0$, which is a contradiction with \eqref{APP:GREEN:eq:psi}. This proves \eqref{APP:GREEN:ineq:lem} when $V$ vanishes around $0$.

\medskip\noindent{\bf Step 2: The general case.} Let $\eta\in C^\infty(\rr)$ be such that $\eta(t)=0$ if $t\leq 1$, $\eta(t)=1$ if $t\geq 2$ and $0\leq\eta\leq 1$. For any $\eps>0$, define $V_\eps(x):=\eta(|x|/\eps)V(x)$ for all $x\in B_{\delta_2}(0)$. Up to taking $\delta_2>0$ small enough to get coercivity, for any $\eps>0$, there exists  $\varphi_\eps\in H_k^2(B_{\delta_2}(0))\cap H_{2k}^q(B_{\delta_2}(0))$ for all $q>1$  such that
\begin{equation}\label{APP:GREEN:eq:phi:eps}
\left\{\begin{array}{ll}
(P-V_\eps)\varphi_\eps=0&\hbox{ in }B_{\delta_2}(0)\\
\partial_\nu^i\varphi_\eps=\partial_\nu^i\varphi&\hbox{ on }\partial B_{\delta_2}(0)\hbox{ for }i=0,\cdots,k-1\\
\end{array}\right.
\end{equation}
As one checks, $\lim_{\eps\to 0}\varphi_\eps=\varphi$ in $H_k^2(B_{\delta_2}(0))$ and $\lim_{\eps\to 0}\varphi_\eps(x)=\varphi(x)$ for all $x\in \bar{B}_{\delta_1}(0)- \{0\}$. Since $V_\eps$ vanishes around $0$, we  apply  \eqref{APP:GREEN:eq:psi} to $\varphi_\eps$ and let $\eps\to 0$. We leave the details to the reader. The estimates on the derivatives are consequence of elliptic theory.

\subsection{Green's function for elliptic operators with bounded coefficients}\label{APP:GREEN:app:green}
\begin{defi}\label{APP:GREEN:def:green} Let $\Omega$ be a smooth bounded domain of $\rn$. Fix $k\in\nn$ such that $n>2k\geq 2$. Let $P$ be an elliptic operator of order $2k$. A \emph{Green's function} for $P$ is a function $(x,y)\mapsto G(x,y)=G_x(y)$ defined for all $x\in \O$ and a.e. $y\in \O$ such that 

\begin{itemize}
\item[(i)] $G_x\in L^1(\O)$ for all $x\in \O$,
\item[(ii)] for all $x\in \O$ and all $\varphi\in C^{2k}(\Obar)$ such that $\partial_\nu^i\varphi_{|\partial\Omega}=0$ for all $i=0,..,k-1$, we have that
$$\int_\Omega G_xP\varphi\, dx=\varphi(x).$$
\end{itemize}
\end{defi}

\begin{theorem}\label{APP:GREEN:th:green:nonsing} Let $\Omega$ be a smooth bounded domain of $\rn$, $n\geq 2$. Fix $k\in\nn$ such that $n>2k\geq 2$ and $L>0$. Let $P$ be an elliptic operator such that \eqref{APP:GREEN:def:okl} holds. Then there exists a unique Green's function for $P$. Moreover,

\begin{itemize}
\item $G$ extends to $\O\times \O\setminus\{(x,x)/x\in \O\}$ and for any $x\in \O$, $G_x\in H_{k,0,loc}^2(\O-\{x\})\cap H_{2k,loc}^p(\O-\{x\})$ for all $p>1$ and $G_x\in C^{2k-1}(\Obar-\{x\})$
\item $G$ is symmetric;
\item For all $x\in \O$, we have that
\begin{equation*}
\left\{\begin{array}{cc}
PG_x=0&\hbox{ in }\O\setminus\{x\}\, \\
{\partial_{\nu}^iG_x}_{|\partial\O}=0&\hbox{ for }i=0,...,k-1.
\end{array}\right\}
\end{equation*}
\item For all $f\in L^p(\O)$, $p>\frac{n}{2k}$, and $\varphi\in H_{2k}^p(\O)\cap H_{k,0}^p(\O)$ such that $P\varphi=f$ weakly, then 
$$\varphi(x)=\int_\O G_xP\varphi\, dx\hbox{ for all }x\in \O.$$
\item For all $\varphi\in C^{2k}(\Obar)$, we have that 
$$\varphi(x)=\int_\O G_xP\varphi\, dy-\int_{\partial\O}C_{P}(\varphi,G_x)\, d\sigma\hbox{ for all }x\in \O.$$
where
$$C_{P}(\varphi,G_x):=-\sum_{2i+1\leq k-1}\partial_\nu\Delta^{i}\varphi\Delta^{k-1-i}G_x+ \sum_{2i\leq k-1}\Delta^{i}\varphi\partial_\nu\Delta^{k-1-i}G_x.$$
If $\partial_\nu^i\varphi=0$ on $\partial\O$ for all $i=0,...,k-1$, then $C_{P}(\varphi,G_x)\equiv 0$ on $\partial\O$.

\item For all $\omega\subset\subset\O$, There exists $C(k,L,\omega)>0$ such that
$$|G_x(y)|\leq C(k,L,\omega)\cdot |x-y|^{2k-n}\hbox{ for all }x\in\omega,\,y\in \Omega, \, x\neq y,$$
\item For all $l=1,...,2k-1$, there exists $C_l(k,L,\omega)>0$ such that
$$|\nabla^l G_x(y)|\leq C_l(k,L,\omega)\cdot |x-y|^{2k-n-l}\hbox{ for all }x\in\omega,\,y\in \Omega, \, x\neq y;$$
\end{itemize}
\end{theorem}

\smallskip\noindent The sequel of this subsection is devoted to the proof of Theorem \ref{APP:GREEN:th:green:nonsing}. We build the Green's function via the classical Neumann series following Robert \cite{robert:green}. Let $\eta\in C^\infty(\rr)$ be such that $\eta(t)=1$ if $t\leq 1/4$ and $\eta(t)=0$ if $t\geq 1/2$. We  define
\begin{equation*}
\Gamma_x(y)=\Gamma(x,y):=C_{n,k}|x-y|^{2k-n}\hbox{ for all }x,y\in \O,\; x\neq y.
\end{equation*}
where $C_{n,k}$ is defined in \eqref{APP:GREEN:def:Cnk}. Note that $\Gamma_x\in C^\infty(\Obar-\{x\})$.

\smallskip\noindent{\bf Step 1:} As in the proof of Step 1.3, see formula \eqref{APP:GREEN:ipp:gamma}, for all $x\in \O$, there exists $f_x\in L^1(\O)$ such that
\begin{equation}\label{APP:GREEN:lap:gamma}
\left\{\begin{array}{cc}
P\Gamma_x=\delta_x-f_x&\hbox{ weakly in  }\O\\
|f_x(y)|\leq C(k,L) |x-y|^{2k-n}&\hbox{ for all }x,y\in \O,\, x\neq y,
\end{array}\right\}
\end{equation} 
Where the equality is to be taken in the distribution sense, that is 
$$\int_\O \Gamma_x P\varphi\, dx=\varphi(x)-\int_\O f_x\varphi\, dx+\int_{\partial \O}\sum_{i=0}^{k-1}\mathcal{B}^{(i)}(\varphi,\Gamma_x)\, d\sigma\hbox{ for all }\varphi\in C^{2k}(\Obar),$$
where the $\mathcal{B}^{(i)}$'s are defined in \eqref{def:B} and where $f_x:=-(\Delta^k\Gamma_x +h\Gamma_x)$.

\smallskip\noindent{\bf Step 2:} We are now in position to define the Green's function $G$. We define
\begin{equation*}
\left\{\begin{array}{cc}
\Gamma_1(x,y):=f_x(y) &\hbox{ for }x,y\in \Omega,\, x\neq y,\\
\Gamma_{i+1}(x,y):=\int_\O \Gamma_i(x,z)f_z(y)\, dz&\hbox{ for }x,y\in \O,\, x\neq y,\, i\in\nn
\end{array}\right\}
\end{equation*}
With straightforward computations (Giraud's Lemma \cite{giraud}, as stated in \cite{DHR} for instance), the definition of $\Gamma$ and \eqref{APP:GREEN:lap:gamma}, for all $i\in\nn$, we have that
\begin{equation}\label{APP:GREEN:ctrl:gamma:i}
|\Gamma_i(x,y)|\leq C_i(k,L,\Omega) \left\{\begin{array}{cc}
|x-y|^{2ki-n} &\hbox{ if }2ki<n;\\
1+|\ln |x-y||&\hbox{ if }2ki=n;\\
1 &\hbox{ if }2ki>n.\end{array}\right.
\end{equation}
for all $x,y\in\Omega$, $x\neq y$. We then get that $\Gamma_i(x,\cdot)\in L^\infty( \O)$ for all $x\in \Omega$ and $i>\frac{n}{2k}$. We fix $p>n/k$. For $x\in \O$, we take $u_x\in H_{2k}^2(\O)\cap C^{2k-1}(\Obar)$ that will be fixed later, and we define
\begin{equation}\label{APP:GREEN:def:G}
G_x(y):=\Gamma_x(y)+\sum_{i=1}^p\int_\O\Gamma_i(x,z)\Gamma(z,y)\, dz+ u_x(y)\hbox{ for a.e }y\in \O.
\end{equation} We fix $\varphi\in C^{2k}(\Obar)$. Via Fubini's theorem, using  the definition of the $\Gamma_i$'s and the definition of $P$, we get that
\begin{eqnarray*}
&&\int_\O G_x P\varphi\, dy = \int_\O \Gamma_x P\varphi\, dy +\sum_{i=1}^p\int_{\O\times\O} \Gamma_i(x,z)\Gamma(z,y)P\varphi(y)\, dzdy\\
&& +\int_\O Pu_x\varphi\, dy +\int_{\partial \O}\sum_{i=0}^{k-1}\mathcal{B}^{(i)}(\varphi,u_x)\, d\sigma\\
&&= \varphi(x)-\int_\O \Gamma_1(x,\cdot)\varphi\, dx +\sum_{i=1}^p\int_{\O} \Gamma_i(x,z) \varphi(z) \,dz\\
&&-\sum_{i=1}^p\int_{\O }\left(\int_\O \Gamma_i(x,z)   f_z(y)\, dz\right)\varphi(y)\,  dy\\
&& +\int_\O Pu_x\varphi\, dx+\int_{\partial \O}\sum_{i=0}^{k-1}\mathcal{B}^{(i)}(\varphi,G_x)\, d\sigma\\
&&= \varphi(x)  +\int_{\O}(Pu_x-\Gamma_{p+1}(x,\cdot))\varphi\,  dy +\int_{\partial \O}\sum_{i=0}^{k-1}\mathcal{B}^{(i)}(\varphi,G_x)\, d\sigma
\end{eqnarray*}
Since $\Gamma_{p+1}(x,\cdot)\in L^\infty(\O)$, we choose $u_x\in \cap_{q>1}H_{2k}^q(\O)\cap H_{k,0}^q(\O) $  such that
$$\left\{\begin{array}{cc}
Pu_x=\Gamma_{p+1}(x,\cdot)&\hbox{  in }\O.\\
\partial_\nu^i u_x=-\partial_\nu^i\left(\Gamma_x+\sum_{i=1}^p\int_\O\Gamma_i(x,z)\Gamma(z,\cdot)\, dz\right)&\hbox{ on }\partial\O
\end{array}\right\}$$
The existence follows from Theorem \ref{APP:GREEN:th:3}. Sobolev's embedding theorem yields $u_x\in C^{2k-1}(\Obar)$ and Theorem \ref{APP:GREEN:th:3} yields  $ C( k,L,p,\omega)>0$ such that
\begin{equation}\label{APP:GREEN:ctrl:u}
|u_x(y)|\leq C(k,L,p,\omega)\hbox{ for all }x\in\omega,\,y\in \Omega.
\end{equation}
In particular, $G_x\in C^{2k-1}(\Obar\setminus\{x\})$ and $\partial_\nu^iG_x=0$ on $\partial\O$ and $i=0,...,k-1$. Finally, we get that 
\begin{equation}\label{APP:GREEN:id:green}
\int_\Omega G_x P\varphi\, dy=\varphi(x)+\int_{\partial \O}\sum_{i=0}^{k-1}\mathcal{B}^{(i)}(\varphi,G_x)\, d\sigma\hbox{ for all }\varphi\in C^{2k}(\Obar).
\end{equation}
Note that since $\partial_\nu^iG_x=0$ on $\partial\O$ and $i=0,...,k-1$, then $\nabla^iG_x=0$ on $\partial\O$ for $i=0,...,k-1$ and then  we have that
\begin{eqnarray*}
\sum_{i=0}^{k-1}\mathcal{B}^{(i)}(\varphi,G_x)&=&  -\sum_{2i+1\leq k-1}\partial_\nu\Delta^{i}\varphi\Delta^{k-1-i}G_x+ \sum_{2i\leq k-1}\Delta^{i}\varphi\partial_\nu\Delta^{k-i-1} G_x
\end{eqnarray*}
The  controls \eqref{APP:GREEN:ctrl:gamma:i} and \eqref{APP:GREEN:ctrl:u}, the definition \eqref{APP:GREEN:def:G} and Giraud's Lemma yield
\begin{equation}\label{APP:GREEN:ctrl:G}
\begin{array}{c}
|G_x(y)|\leq C( k, L,\omega)|x-y|^{2k-n}\\
\end{array}\ \hbox{ for all }x\in\omega,y\in \Omega, \, x\neq y.
\end{equation}
This proves the existence of a Green's function for $P$. Moreover, the construction yields $G_x\in H_{2k,loc}^p (\O-\{x\})\cap H_{k,0, loc}^p (\O-\{x\})$ for all $p>1$ and $P G_x=0$  in $\Omega-\{x\}$. The validity of \eqref{APP:GREEN:id:green} for $u\in H_{2k}^p(\O)\cap H_{k,0}^p(\O)$ and $f\in L^p(\O)$ such that $Pu=f$ and $p>n/(2k)$ follows by density of $C^\infty_c(\O)$ in $L^p(\O)$ and the regularity Theorem \ref{APP:GREEN:th:3}. The symmetry of $G$ follows from the self-adjointness of the operator $P$. The uniqueness goes as the proof of uniqueness of Theorem \ref{APP:GREEN:th:Green:main}. The pointwise control for $|G_x(y)|$ is  \eqref{APP:GREEN:ctrl:G}. The control of the gradient of $G_x$ is a consequence of elliptic theory. Since the details of these points are exactly the same as in the case of a second-order operator $\Delta+h$, we refer to the detailed construction \cite{robert:green}.
\section{Regularity theorems}\label{APP:GREEN:sec:regul:adn}
The following theorems are reformulations  of Agmon-Douglis-Nirenberg \cite{ADN}.

\begin{theorem}\label{APP:GREEN:th:2:again} We fix $k\in\nn$, $L>0$ and $\delta>0$. Let $\O$ be a smooth domain of $\rn$, $n>2k\geq 2$ and $x_0\in\Obar=\O\cup\partial\Omega$. Let $P=\Delta^k+h$ be a differential operator such that $h\in L^\infty(\O\cap B_\delta(x_0))$ and $\Vert h\Vert_\infty\leq L$. Let $u\in H_{2k}^s(\Omega\cap B_\delta(x_0))$ be such that $\eta u\in H_{k,0}^s(\O)$ for all $\eta\in C^\infty_c(B_\delta(x_0))$ and $f\in L^p(\O\cap B_\delta(x_0))$, $p,s\in (1,+\infty)$ be such that $Pu=f$. Then for all $r<\delta$, $u\in H_{2k}^p(\Omega\cap B_r(x_0))$. Moreover, for all $q>1$, we have that
$$\Vert u\Vert_{H_{2k}^p(\O\cap B_r(x_0))}\leq C(n,\O,k,L,p,q,\delta,r)\left(\Vert f\Vert_{L^p(\O\cap B_\delta(x_0))}+\Vert u\Vert_{L^q(\O\cap B_\delta(x_0))}\right)$$
where $C(n,\O,k,L,p,q,\delta,r)$ depends only on $n$, $\O$, $k$, $L$, $p$, $q$, $\delta$ and $r$.
\end{theorem}

\begin{theorem}\label{APP:GREEN:th:2:holder} We fix $k\in\nn$ and $L>0$ and $\delta>0$. Let $\O$ be a smooth domain of $\rn$, $n>2k\geq 2$ and $x_0\in\Obar=\O\cup\partial\Omega$. Let $P=\Delta^k+h$ be a differential operator such that $h\in C^{0,\alpha}(\O\cap B_\delta(x_0))$ and $\Vert h\Vert_{C^{0,\alpha}}\leq L$ for some $\alpha\in (0,1)$. Let $u\in C^{2k,\alpha}(\Omega\cap B_\delta(x_0))$ be such that $\partial_\nu^iu=0$ on $B_\delta(x_0)\cap \partial\O$ for all $i=0,...,k-1$ and $f\in C^{0,\alpha}(\O\cap B_\delta(x_0))$ be such that $Pu=f$. Then for all $r<\delta$, we have that
$$\Vert u\Vert_{C^{2k,\alpha}(\O\cap B_r(x_0))}\leq C(n,\O,k,L,\alpha,\delta,r)\left(\Vert f\Vert_{C^{0,\alpha}(\O\cap B_\delta(x_0))}+\Vert u\Vert_{C^0(\O\cap B_\delta(x_0))}\right)$$
where $C(n,\O,k,L,\alpha,\delta,r)$ depends only on $n$, $\O$, $k$, $L$, $\alpha$, $\delta$ and $r$.
\end{theorem}

\begin{theorem}\label{APP:GREEN:th:3} We fix $k\in\nn$ and  $L>0$. Let $\O$ be a smooth domain of $\rn$, $n>2k\geq 2$. Let $P$ be a differential operator such that \eqref{APP:GREEN:def:okl} holds and fix $p\in (1,+\infty)$. Then for all $f\in L^p(\O)$, there exists $u\in H_{2k}^p(\O)\cap H_{k,0}^p(\O)$ unique such that $Pu=f$. Moreover, for some $C(\O,k,L,p)$ depends only on $\O$, $k$, $L$ and $p$, we have that
$$\Vert u\Vert_{H_{2k}^p(\O)}\leq C(\Omega,k,L,p ) \Vert f\Vert_{L^p(\O)}.$$
\end{theorem}

\end{document}